\documentclass[12pt]{article}
\usepackage{amssymb,amsmath}
\input{amssym.def}\input{amssym}
\usepackage{graphicx, tikz}

\newcommand{\C}{{\mathbb C}}
\newcommand{\D}{{\mathbb D}}

\newcommand{\N}{{\mathbb N}}

\newcommand{\R}{{\mathbb R}}

\newcommand\cB{{\mathcal B}}
\newcommand\cC{{\mathcal C}}

\newcommand\cH{{\mathcal H}}

\newcommand\cN{{\mathcal N}}

\newcommand\cW{{\mathcal W}}
\newcommand\cY{{\mathcal Y}}
\newcommand\cP{{\mathcal P}}
\newcommand\cS{{\mathcal S}}

\newcommand\bD{{\mathbb D}}
\newcommand\bE{{\mathbb E}}

\newcommand\bN{{\mathbb N}}

\newcommand\bR{{\mathbb R}}

\newcommand\bT{{\mathbb T}}

\newcommand\Id{{\bf 1}}

\newtheorem{theorem}{Theorem}[section]
\newtheorem{prop}{Proposition}[section]

\newtheorem{lemma}{Lemma}[section]
\newtheorem{remark}{Remark}[section]
\newtheorem{definition}{Definition}[section]
\newenvironment{proof}{\noindent {\bf Proof.}}{ \hfill $\Box$\\ }
\newenvironment{proofof}[1]{\noindent {\bf Proof of #1.}}{ \hfill $\Box$\\ }

\newcommand\ve{\varepsilon}

\newcommand\vf{\varphi}

\title{Regular variation and rates of mixing for infinite measure preserving almost Anosov diffeomorphisms}
\author{Henk Bruin
\thanks{Faculty of Mathematics, University of Vienna, 
Oskar Morgensternplatz 1, 1090 Vienna, Austria; {\it henk.bruin@univie.ac.at}}
 \and Dalia Terhesiu
\thanks{Department of Mathematics, Harrison Building Streatham Campus, 
University of Exeter, North Park Road, Exeter EX4 4QF, UK; {\it daliaterhesiu@gmail.com}}
}
\date{January 2018}

\begin{document}

\maketitle

\abstract{The purpose of this paper is to establish mixing rates for infinite measure preserving almost Anosov diffeomorphisms
on the two-dimensional torus. The main task is to establish regular variation of the tails
of the first return time to the complement of a neighbourhood of the neutral fixed point.
}

\section{Introduction}
Almost Anosov diffeomorphisms
on the two-dimensional torus were introduced in~\cite{HY95, Hu00}  
where it was shown that, depending on the local dynamics near the (for simplicity single)
neutral fixed point, there is a finite or infinite SRB measure.
The recent paper \cite{HuZhangprepr} provides polynomial mixing rates,
in the finite measure setting and with bounds on the exponent rather than a precise exponent,
within a certain parameter range, see Remark~\ref{rem:HZ}.
In \cite{HuZhangprepr}, the existence of a Markov partition, a first return
inducing scheme, quotient dynamics (i.e., they factor out the stable direction),
and the by now well-established renewal theory \cite{Sarig02, Gouezel04} are used.

In this paper we prove mixing results for almost Anosov diffeomorphisms on the torus,
with an infinite SRB measure. The class of maps is somewhat wider (see \eqref{eq:diffeo} below),
but more importantly, using an entirely different method from \cite{Hu00, HuZhangprepr}, we are able to
establish much more precise tail estimates for the inducing scheme: there is $C_0>0$ such that
\begin{equation}\label{eq:tail}
\mu(\varphi > n) =  C_0 n^{-\beta} (1+o(1))
\end{equation}
with the more precise form the $o(1)$ given in Theorems~\ref{thm:regvarmu} and \ref{thm:regvarmu_special}.
We say that the sequence $\mu(\varphi > n)$ is regularly varying with index $\beta$ with 
involved slowly varying function $\ell(n) = C_0$.
These estimates hold for the entire parameter range $\{ a_0,a_2,b_0,b_2 > 0, a_0b_2-a_2b_0 \neq 0\}$,
see formula~\eqref{eq:diffeo}, so both for the finite measure ($\beta  1$)
and the infinite measure case ($\beta \leq 1$).
In the finite measure setting we limit ourselves to a few remarks, because the equality in \eqref{eq:tail}
 is not needed to obtain mixing rates or limit laws (including $\beta$-stable laws for $\beta \in (1,2]$). 
In this paper, we focus our attention on the infinite measure setting where regular variation of the 
tail of $\varphi$ is crucial to obtain mixing results
as presented in Theorems~\ref{thm:mix1}, \ref{thm:mix2} and \ref{thm:mix3}.

A way of obtaining mixing rates for finite measure preserving invertible Markov systems is
to collapse stable leaves, work on the quotient space and transfer the results back to
the original system.
In the infinite measure setting this strategy gives mixing, but doesn't seem to work
for mixing rates (that is, higher order of the correlation function, \cite{MT}, see also 
Theorem~\ref{thm:mix2} below for an illustration of this notion): see \cite{Mel15}.

For this reason we use, as in \cite{LT},
anisotropic Banach spaces of distributions where the transfer operator of the induced map acts,
and we establish the required spectral properties directly.
More precisely, we show that a slightly modified version
of the Banach space considered in~\cite{LT} (a simplification of~\cite{DemersLiverani08}) can be used to obtain optimal rates 
of mixing for almost Anosov diffeomorphisms preserving an infinite SRB measure studied in~\cite{Hu00}.
\\[3mm]
{\bf Acknowledgements:} Both authors are grateful to Viviane Baladi for many inspiring 
discussions and careful reading of a previous version of this text.
DT would like to thank CNRS for enabling her a three month visit to
IMJ-PRG, Pierre et Marie Curie University.
HB would like to thank Freddy Dumortier for his explanations of \cite{DRR}
and related literature.
He would also like to thank A\"OU (Aktion \"Osterreich-Ungarn) 
and P\'eter Imre Toth for discussions during a visit to Budapest funded by A\"OU project 92\"ou6.
Finally, both authors would like to thank the Erwin Schr\"odinger Institute
where the final version of this paper was prepared during a ``Research in Teams'' project.

\subsection{Set-up}\label{sec:set-up}

\begin{definition}{~\cite[Definition 1]{Hu00}}
\label{def-AlmAn} 
Let $\bT^2$ be te two-dimensional torus.
A diffeomorphism $f:{\bT}^2\to {\bT}^2$ is called {\em almost Anosov} if there exists two continuous 
families of non-trivial cones $x\to\cC_x^u, \cC_x^s$ such that except for a finite set $S$,
\begin{itemize}
\item[i)] $D f_x\cC_{x}^u\subseteq \cC_{f(x)}^u$ and $D f_x\cC_{x}^s\supseteq \cC_{f(x)}^s$;
\item[ii)] $|D f_x v|>|v|$ for any $0 \neq v\in \cC_x^u$ and $|D f_x v|<|v|$ for any $0 \neq v\in \cC_x^s$.
\end{itemize}
For $x \in S$, $Df_x = I$ is the identity. 
\end{definition}

\begin{remark}\label{rem:add_def}
 In our setting the families of cones will be transversal, i.e., $\cC_x^u \cap \cC_x^s = \{ 0 \}$
 at every $x \in {\bT}^2$, unlike the systems studied in \cite{LM05}
 where $Df_p = \begin{pmatrix} 1 & 1 \\ 0 & 1 \end{pmatrix}$ at the neutral fixed point $p$
 and the corresponding stable and unstable manifolds are in fact tangent.
\end{remark}

The starting point is an almost Anosov Markov diffeomorphism $f:{\bT}^2 \to {\bT}^2$, 
with a single neutral fixed point $p$
and $f$ is $C^{\kappa+2}$-smooth (with even $\kappa \in \N$ fixed), except possibly at
the local stable and unstable manifolds $W^{u/s}(p)$ where only $C^1$ 
smoothness\footnote{but on eather side of $W^{u/s}$ we can extend $f$ to $W^{u/s}$ in a $C^{\kappa+2}$ manner} 
is assumed.
Let $\{ P_i \}_{i = 0}^k$ be the Markov partition for $f$ (which we can assume to exist since $f$ is a local
perturbation from a Anosov diffeomorphism on $\bT^2$).
We assume that $p$ belongs to the interior of $P_0$ and use a system of coordinates in which $p$ is the origin, and
the diffeomorphism $f(x,y)$ has the local form
\begin{equation}\label{eq:diffeo}
\left(x(1+a_0x^\kappa+a_2y^\kappa + O( |(x,y)|^{\kappa+1})), 
y(1-b_0x^\kappa-b_2y^\kappa + O( |(x,y)|^{\kappa+1})) \right),
\end{equation}
on a neighbourhood $U \supset P_0$. 
Here $a_0, a_2, b_0, b_2$ are positive real constants with $\Delta := a_2b_0 - a_0b_2 \neq 0$.
(This notation is taken from \cite{Hu00}; the absence of mixed terms and coefficients
$a_1, b_1$ is explained further in Remark~\ref{rem:mixed}.)
Then the horizontal and vertical axes are the unstable and stable manifolds of $p$ respectively.
We assume that the Markov partition element $P_0 \subset U$ is a small rectangle such that 
$\overline{f^{-1}(P_0) \cup P_0 \cup f(P_0)} \subset U$.
Due to the symmetries $(x,y) \mapsto (\pm x, \pm y)$,
it suffices to do the analysis only in the first quadrant 
$Q = [0,\zeta_0] \times [0,\eta_0]$ of $P_0$, see Figure~\ref{fig:leaves}.
Without loss of generality (see Lemma~\ref{lem:straight})
we can think of $[0, \zeta_0] \times \{ \eta_0 \}$ as a local unstable leaf
and $\{ \zeta_0\} \times [0, \eta_0]$ as a local stable leaf of the global diffeomorphism.

\begin{figure}[ht]
\begin{center}
\begin{tikzpicture}[scale=1.5]
\node at (4.3,-0.1) {\small $x$}; \node at (2.8,-0.2) {\small $\zeta_0$}; \node at (4,-0.2) {\small $\zeta_1$}; 
\node at (-0.1,4.3) {\small $y$}; \node at (-0.2,2.8) {\small $\eta_0$};\node at (-0.2,4) {\small $\eta_1$};
\draw[->, draw=red] (-0.5,0)--(4.4,0);
\draw[->, draw=blue] (0,-0.5)--(0,4.4);
\node at (0.73,3.4) {\tiny $\{ \varphi = n\}$};
\node at (3.4,0.5) {\tiny $F(\{\varphi = n\})$};
\node at (1.5,1.5) {$Q$};
\node at (1.5,2.7) {\tiny $W^u$}; \node at (1.2,4) {\tiny $f^{-1}(W^u)$};
\node at (2.98,1.3) {\tiny $W^s$}; \node at (4.3,0.8) {\tiny $f(W^s)$};
%
\draw[-, draw=red] (0,4)--(0.8,4); \draw[-, draw=blue] (4,0)--(4,0.8); 
\draw[-, draw=blue] (0.1,4) .. controls (0.11,2.5) and (0.4,1.5) .. (0.5,-0.2);
\draw[-, draw=blue] (0.2,4) .. controls (0.21,2.5) and (0.5,1.5) .. (0.6,-0.2);
\draw[-, draw=blue] (0.3,4) .. controls (0.31,2.5) and (0.6,1.5) .. (0.7,-0.2);
\draw[-, draw=blue] (2.8,2.8)--(2.8,-0.0);
\draw[-, draw=red] (4, 0.1) .. controls (2.5, 0.11) and (1.5, 0.4) .. (-0.2, 0.5);
\draw[-, draw=red] (4, 0.2) .. controls (2.5, 0.21) and (1.5, 0.5) .. (-0.2, 0.6);
\draw[-, draw=red] (4, 0.31) .. controls (2.5, 0.32) and (1.5, 0.61) .. (-0.2, 0.72);
\draw[-, draw=red] (2.8,2.8)--(-0.0,2.8);
\end{tikzpicture}
\caption{The first quadrant $Q$ of the rectangle $P_0$, with stable and unstabe foliations
drawn vertically and horizontally, respectively.}
\label{fig:leaves}
\end{center}
\end{figure}
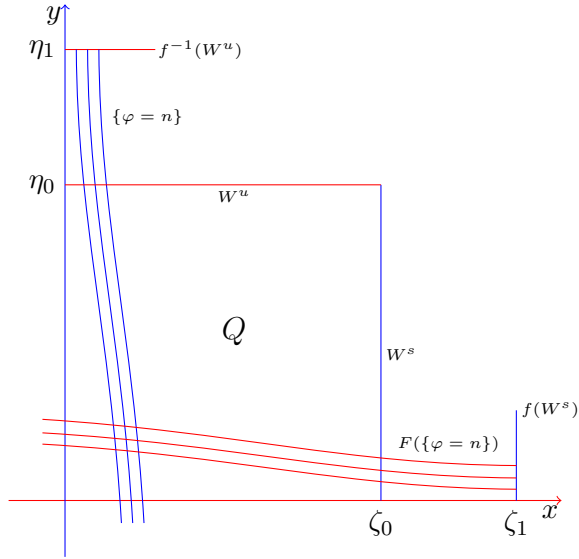 

We consider an induced map $F = f^\varphi: Y \to Y$ for $Y := {\bT}^2 \setminus P_0$, where
$$
\varphi(z) = \min\{ n \geq 1 : f^n(z) \notin P_0\}
$$
is the first return time to $Y$. Note that $F$ is invertible because $f$ is.
In the first quadrant of $U \setminus P_0$, $\{ \varphi = n \} := \{ z \in f^{-1}(Q) \setminus Q : \varphi(z) = n\}$, 
$n \geq 2$, are vertical strips 
adjacent to the local unstable leaf $[0,\zeta_0] \times \{\eta_0 \}$, and 
converging to $\{ 0 \} \times [\eta_0, \eta_1]$
as $n \to \infty$. The images
$F(\{ \varphi = n\})$ are horizontal strips, adjacent to the local stable leaf $\{ \zeta_0 \} \times [0,\eta_0]$, 
and converging to $[\zeta_0, \zeta_1] \times \{ 0 \}$ as $n \to \infty$,
see Figure~\ref{fig:leaves}.

In contrast to $f$, the induced map $F$ is uniformly hyperbolic, but only piecewise continuous.
Indeed, continuity fails at the boundaries of the strips $\{ \varphi = n\}$, $n \geq 2$,
but these boundaries are local stable and unstable leaves,
and it is possible to create a countable Markov partition refining
$\{ P_i \}_{i=1}^k$ of $Y$ for $F$,
in which all the strips $\{ \varphi = n \}$ are partition elements.

The map $F$ preserves an SRB measure $\mu$ and for every maximal unstable leaf $W^u$ inside
$F(P_i)$ for some $i$, the conditional measure
$\mu_{W^u}$ is absolutely continuous w.r.t.\ the conditional Lebesgue measure $m_{W^u}$,
see \cite[Lemma 5.2]{HY95}.
In fact, the density $h_{W^u} := \frac{d\mu_{W^u}}{dm_{W^u}}$ is bounded, bounded away 
from zero and (using a straightforward adaptation of \cite[Proposition 3.1]{HY95})
$C^{\kappa+1}$ smooth.

Our goal is to estimate the tails
$$
\mu(\varphi > n) := \mu(\{z \in f^{-1}(Q) \setminus Q: \varphi(z) > n\}).
$$
The $f$-invariant pullback measure $\mu_f(A) = \sum_{n\geq 0} \mu(f^{-n}(A) \cap \{\varphi > n\})$
is finite if and only if $\sum_n \mu(\varphi > n)< \infty$.
When later on we speak of the (in)finite measure case, this refers to $\mu_f$ being (in)finite,
since $\mu$ itself is always finite, and scaled to be a probability measure.

To estimate $\mu(\varphi > n)$, we first estimate $m(\varphi > n)$ 
(in Proposition~\ref{prop:regvar})
assuming that locally, $f$ is the time-$1$ map
of the flow of the differential equation
\begin{equation}\label{eq:vfhot}
\begin{cases}
\dot x = x(a_0x^\kappa+a_2 y^\kappa + O(|(x,y)|^{\kappa+1}) ), \\
\dot y =  -y(b_0x^\kappa+b_2y^\kappa + O(|(x,y)|^{\kappa+1}) ),
\end{cases}
\end{equation}
where $a_0, a_2, b_0, b_2 > 0$ with $\Delta = a_2b_0 - a_0b_2 \neq 0$. are as before.
This approach of replacing a map by the time-$1$ map of a flow, in the context of billiards, 
was followed in \cite{K79, Machta83, PSZ16} and \cite[Section 6]{BCD11}
but only insofar that the time-$1$ map of the flow approximates the map.
In general, it is unlikely that  on the whole manifold a diffeomorphism exactly coincides with
the time-$1$ map of a vector field, see e.g.\ \cite{Palis}.
However, due to \cite[Theorem B and consequence 1(ii) on page 36)]{DRR}, inside the neighbourhood $U$ 
this holds for $C^\infty$ diffeomorphisms. 
Proposition~\ref{prop:perturb} yields the regular variation of
the tails w.r.t.\ Lebesgue measure in this $C^\infty$ setting.
For $C^{\kappa+2}$ diffeomorphisms $f$, there seems to be no general result that
$f$ is the time-$1$ map of a $C^{\kappa+1}$ vector field 
(\cite{DR} gives results depending on the smoothness of the vector field). 
Therefore, we first perturb $f$
to a $C^\infty$ diffeomorphism, next 
use the asymptotics obtained from the vector field,
and finally interpret the results for $f$ using an extra approximation argument,
see Lemma~\ref{lem:xi}.
The fact that the conditional densities $h_{W^u}$ are smooth finally
allows us to transfer the Lebesgue estimates to estimates in terms of $\mu$. 

\subsection{Main results}\label{sec:mainresults}

\begin{theorem}\label{thm:regvarmu}
Let $f$ be a $C^{\kappa+2}$ diffeomorphism (with even $\kappa \in \N$) on the torus of the form \eqref{eq:diffeo}
near $(0,0)$. Then the tails $\mu(\varphi > n)$ have regular variation:
$$
\mu(\varphi > n) = C_0 n^{-\beta} (1 + O(\max \{ n^{-\beta_*}, n^{-1}\log n\}) ),
$$
where $\beta := \frac{a_2+b_2}{\kappa b_2}$, 
$\beta_* := \frac{1}{\kappa}\min\left\{1,\frac{a_2}{b_2}, \frac{b_0}{a_0}\right \} \leq \min\{ \beta, 1/\kappa\}$
and $C_0>0$ is a constant given at the end of the proof.
In particular, $\mu_f$ is infinite if and only if $\beta \leq 1$.
\end{theorem}

\begin{remark}
The $C^{\kappa+2}$ assumption is used 
for metric estimates in Proposition~\ref{prop:perturb}.
As $\kappa \geq 2$, this implies that $f$ is $C^4$,
and this is used to guarantee the existence of the SRB-measure $\mu_f$
(which relies on $C^4$ smoothness of the stable and unstable leaves,
see \cite[Theorem B]{Hu00}).
\end{remark}



In a special case (namely, when $f$ is the time-$1$ map of a vector field
without higher order terms), our estimates are much sharper:

\begin{theorem}\label{thm:regvarmu_special}
Suppose that the $C^4$ almost Anosov diffeomorphism $f$ on the torus
is near $(0,0)$ the time-$1$ map of the flow of  
\begin{equation}\label{eq:vf}
\begin{cases}
\dot x = x(a_0x^\kappa+a_2y^\kappa), \\
\dot y = -y(b_0x^\kappa+b_2 y^\kappa),
\end{cases}
\end{equation}
where $a_0, a_2, b_0, b_2$ are positive real constants with 
$\Delta := a_2b_0 - a_0b_2 \neq 0$.
Then there are real constants\footnote{These constants are typically nonzero; the precise values are 
given in the proof of Theorem~\ref{thm:regvarmu_special}.} $H_j$ and $\hat H_j$ such that 
$$
 \mu(\varphi > n) = \sum_{j=1}^\kappa H_j n^{-j \beta} - \sum_{j=1}^\kappa \hat H_j n^{-(j\beta+1)}
 + O(\max\{ n^{-(\kappa+1)\beta}, n^{-(2+\beta)} \} )
$$
for all $n \geq 2$. As before, $\beta = \frac{a_2+b_2}{\kappa b_2}$.
\end{theorem}

Theorems~\ref{thm:regvarmu} and \ref{thm:regvarmu_special} (proved at the end of 
Sections~\ref{sec:proof} and \ref{sec:asymp} respectively)
strengthen the estimates in \cite{HuZhangprepr} considerably,
see Remark~\ref{rem:HZ}.

\begin{remark}\label{rem:LebesgueSRB}
 If $(\kappa+1) a_0 = b_0$ and $a_2 = (\kappa+1) b_2$, then the divergence of the vector field
 in \eqref{eq:vf} is zero, and the flow preserves Lebesgue measure (cf.\ \cite{K79}).
 In this case $\beta = \frac{\kappa+2}{\kappa} > 1$.
 Provided the global dynamics preserves Lebesgue as well, Lebesgue measure is the SRB measure
 both in forward and backward time.
\end{remark}

\begin{remark}
The exponents seen in Theorems~\ref{thm:regvarmu} and \ref{thm:regvarmu_special}
depend on the quotients $b_2/a_2$ and $b_0/a_0$ rather than on $a_0, a_2, b_0, b_2$
separately. To explain this, note that the linear change of coordinates
 $r \bar x = x, \ s \bar y = y$ (on the quadrant $Q$, so we can drop the absolute value signs)
 transforms \eqref{eq:vf} into
 $$
 \begin{cases}
\dot{\bar x} = \bar x(a_0 r^\kappa \bar x^\kappa+a_2 s^\kappa \bar  y^\kappa), \\
\dot{\bar y} = -\bar y(b_0 r^\kappa \bar x^\kappa+b_2 s^\kappa \bar y^\kappa).
\end{cases}
$$
Apart from changing the size of the rectangle $Q$, this
has no effect on the asymptotic behavior of the 
system (such as those described in Proposition~\ref{prop:regvar}), and therefore
the important parameters (i.e., exponents) in the asymptotics should be independent of $r,s$.
This is indeed true when they depend only on the quotients $b_2/a_2$ and $b_0/a_0$.
\end{remark}


\begin{theorem}\label{thm:mix1}
Consider the setting of Theorem~\ref{thm:regvarmu}
and assume that $\beta \in (\frac12,1)$.
 Then for all observables $v, w \in C^1$ supported on $Y$ we have
 \begin{equation}\label{eq:mix1}
  \lim_{n \to \infty} n^{1-\beta} \int_{\bT^2} v \cdot w \circ f^n \, d\mu
  = d_0 \int_{\bT^2} v \, d\mu_f \int_{\bT^2} w \, d\mu,
 \end{equation}
 where $d_0 = \left( C_0\, \Gamma(\beta)\, \Gamma(1-\beta) \right)^{-1}$ 
 with $C_0 > 0$ as in Theorem~\ref{thm:regvarmu}.
\end{theorem}

\begin{remark}\label{rem:beta1}
 The case $\beta = 1$ and $\liminf$ results for $\beta \leq \frac12$ can be obtained as in \cite{MT}.
Because the proofs here follow the structure of the arguments in \cite{LT}, and the case $\beta = 1$ 
was omitted there, we will omit them here as well. For the range $\beta \leq \frac12$,
we provide a stronger result in Theorem~\ref{thm:mix2}.
\end{remark}

\begin{remark}\label{rem:mix}
It is not clear to us how to get good estimates for the {\em small tails} $\mu(\varphi = n)$.
Theorem~\ref{thm:mix1} could have been improved to all $\beta \in (0,1)$
provided the small tails satisfy $\mu(\varphi = n) = O(n^{-(1+\beta)})$:
this would then have been an immediate consequence of \cite{Gouezel11}.
However, in the setting of Theorem~\ref{thm:regvarmu}, we have no better estimate than
$\mu(\varphi = n) = O(\max\{ n^{-(\beta+\beta_*)}, n^{-(1+\beta)} \log n \})$ and $\beta_* \leq \beta$.
This is caused by the estimates for $\tilde T$ we use in the last step of the proof
of  Proposition~\ref{prop:perturb}.
\end{remark}

Mixing rates are obtained as well.

\begin{theorem}\label{thm:mix2}
 Assume the setting of Theorem~\ref{thm:mix1} with\footnote{There is an open region in parameter space 
(e.g.\ $\frac23 b_2 < a_2 < b_2$, $b_0 > \frac23 a_0$ for $\kappa = 2$),
where $\beta_* \in (\frac13, \frac12)$, and $\beta + \beta_* \in (\frac76,\frac32)$.} 
$\beta \in (\frac12,1)$ and $\beta_* \in (\frac13,\beta)$.
 Take $q := \max\{ j \in \N :  (j+1)\beta > j\}$.
 Then there are  (generically nonzero) real constants $d_1, \dots , d_q$ such that
 \begin{equation}\label{eq:mix2}
  \int_{\bT^2} v \cdot w \circ f^n \, d\mu =
  \left( d_0 n^{\beta-1} + d_1 n^{2(\beta-1)} + \dots + d_q n^{(q+1)(\beta-1)} \right)
   \int_{\bT^2} v \, d\mu \int_{\bT^2} w \, d\mu + E_n,
 \end{equation}
where $E_n = O( n^{-(\beta_*-1/3)}, n^{-(\beta + \beta_* -1)/3})$.
\end{theorem}

In the stronger setting of Theorem~\ref{thm:regvarmu_special}, the rates can be improved.

\begin{theorem}\label{thm:mix3}
 Assume the setting of Theorem~\ref{thm:mix2} for the time-$1$ map of the 
 vector field \eqref{eq:vf} and $\beta \in (0,1)$.
 Then \eqref{eq:mix2} holds with $E_n = O(n^{-1}(\log n)^r)$ for $r = 1$ if $\beta \neq \frac12$ 
 and $r = 2$ if $\beta = \frac12$.
 \end{theorem}

\subsection{Summary of the abstract framework and hypotheses in~\cite{LT}}\label{sec-LT}

In this section we recall the abstract framework and hypotheses in~\cite{LT} as relevant 
for the class of maps previously described. We check these hypotheses for our
almost Anosov diffeomorphisms in Section~\ref{sec-Bspace}.

Let $R:L^1(m)\to L^1(m)$ be the transfer operator for the first return 
map $F:Y\to Y$ w.r.t.\ Lebesgue measure $m$. 
We recall that $R$ is defined by duality on $L^1(m)$ via the formula
$\int_Y Rv\,w\,d m = \int_Y v\,w\circ F\,d m$ for all bounded and measurable $w$ 
(see Section~\ref{subsec-Trop} for a more explicit description).

In this work we verify that in the case of the map described in Section~\ref{sec:set-up}, 
there exist two Banach spaces of distributions $\cB$, $\cB_w$ supported on $Y$ such that

\begin{itemize}\label{H1}
\item[(H1)] \begin{itemize}
\item[(i)] $C^1\subset\cB\subset\cB_w\subset (C^1)'$, where 
$(C^1)'$ is the dual of $C^1 = C^1 (Y,\C)$.\footnote{We will use systematically a ``prime" to denote the topological dual.}
\item[(ii)] There exists $C>0$ such that for all $h\in\cB$, $\phi\in C^1$, we have $h\phi\in\cB$ and 
$\|h\phi\|_{\cB}\leq C\|h\|_{\cB} \|\phi\|_{C^1}$.
\item[(iii)] The transfer operator $R$ associated with $F$ mapping continuously from $C^1$ to $\cB$, and
$R$ admits a continuous extension to an operator from $\cB$ to $\cB$, which we still call $R$.
\item[(iv)] The operator $R:\cB\to\cB$ has a simple eigenvalue at $1$
and the rest of the spectrum is contained in a disc of radius less than $1$.
\end{itemize}
\end{itemize}

A few comments on (H1) are in order and here we just recall the ones in~\cite{LT}.
We note that (H1)(i) should be understood in terms of the usual convention  (see, for instance,
~\cite{GL06, DemersLiverani08}) which we follow here: there exists continuous injective linear maps $\pi_i$ such that $\pi_1(C^1)\subset \cB$, $\pi_2(\cB)\subset \cB_w$ and $\pi_3(\cB_w)\subset (C^1)'$.
We will often leave such maps implicit, unless this might create confusion. In particular, 
we assume that $\pi=\pi_3\circ\pi_2\circ \pi_1$ is the usual 
embedding, i.e., for all $h, \phi \in C^1$
$$
\langle \pi(h),\phi\rangle=\int_Y h \phi\, dm.
$$
Via the above identification, the Lebesgue measure $m$ can be identified with the constant function $1$ both in $(C^1)'$ and in $\cB$ 
(i.e., $\pi(1)=m$). 
Also, by (H1)(i), $\cB'\subset (C^1)'$, hence the dual space can be naturally viewed as a space of distributions. Next, note that $\cB'\supset (C^1)''\supset C^1\ni 1$, thus we have $m\in\cB'$ as well.
Moreover, since $m\in\cB$ and $\langle1,\phi\rangle=\langle\phi,1\rangle=\int \phi\, dm$, $m$ can be viewed as the element $1$  of both spaces
$\cB$ and $(C^1)'$.

By (H1), the spectral projection $P$ associated with the eigenvalue $1$ is defined by
$P=\lim_{n\to\infty}R^n$. Note that for each $\phi\in C^1$,
\[
\langle P\phi,1\rangle=m(P \phi)=\lim_{n\to\infty}m(1\cdot R^n\phi)=m(\phi)=\langle\phi,1\rangle.
\]
By  (H1)(iv),  there exists a unique $\mu\in\cB$ such that $R\mu=\mu$ and $\langle\mu,1\rangle =1$ . Thus, $P \phi=\mu \langle \phi,1\rangle$.
Also $R'm=m$ where $R'$ is dual operator acting on $\cB'$. Note that for any $\phi\in C^1$,
\begin{align*}
|\langle \mu, \phi\rangle|=|\langle P1, \phi\rangle|=\left|\lim_{n\to\infty}R^n m(\phi)\right|
= \lim_{n\to\infty}\left|m(\phi\circ F^n)\right|\leq |\phi|_\infty.
\end{align*}
Hence $\mu$ is a measure. For each $\phi\geq 0$, 
\[
\langle P1, \phi\rangle=\lim_{n\to \infty} \langle R^n1, \phi\rangle=\lim_{n\to \infty} \langle 1, \phi\circ F^n\rangle\geq 0.
\]
It follows that $\mu$ is a probability measure.

Summarizing the above, the eigenfunction associated with the eigenvalue $1$
 is an invariant probability measure for $F$ and  we can write $P1=\mu$.
 
In the rest of this section we list the hypotheses of~\cite{LT} that we will verify 
for the map described in Section~\ref{sec:set-up}.
Recall that $\varphi:Y\to \bN$ is the first return time to $Y$. We verify that
\begin{itemize}\label{H2}
\item[(H2)]  
there exists $C>0$, $\alpha\in(0,1]$ such that
for any $n\in\bN$ and $h\in\cB$ we have $\Id_{\{ \varphi = n\}} h\in\cB_w$ and
\[
|\langle \Id_{\{ \varphi = n\}} h, 1\rangle|\leq C \| h\|_{\cB}\;\mu(Y_n)^\alpha.
\]
\end{itemize}

In the \emph{infinite} measure setting Theorem~\ref{thm:regvarmu} implies that

\begin{itemize}
\item[(H3)]
$\mu(y\in Y:\varphi(y)>n)=\ell(n)n^{-\beta}$ where $\ell$ is slowly varying and
$\beta\in(0,1)$.
\end{itemize}

We also verify some standard assumptions in operator renewal theory as first developed in~\cite{Sarig02,Gouezel04} 
(finite measure case) and~\cite{Gouezel11,MT} (infinite measure case). 

Let $\D=\{z\in\C:|z|<1\}$ and
$\bar\D=\{z\in\C:|z|\le1\}$.
Given $z\in\bar\D$, we define the perturbed transfer operator $R(z)$ (acting as operator 
$\cB\to\cB$ or $\cB\to\cB_w$) by
$R(z)v=R(z^\varphi v)$.   Also, for each $n\ge1$, we define
$R_n$ (acting on $\cB, \cB_w$) by  $R_nv=R(1_{\{\varphi=n\}}v)$.
We will show that 

\begin{itemize}
\item[(H4)] $\|R_n\|_{\cB} = O( m(\vf=n) )$.
\end{itemize}

Assumption (H4) ensures that $R(z)=\sum_{n=1}^\infty R_nz^n$ is a well defined family of operators from
$\cB$ to $\cB$ (or from $\cB$ to $\cB_w$). 
Also, we notice that  (H1) and (H4) ensure that $z\mapsto R(z)$,  $z\in\bar\D$,  is a continuous family of bounded operators (analytic on $\D$) from $\cB$ to $\cB$ (or from $\cB$ to $\cB_w$).
The following assumption was essential for the main abstract results in~\cite{LT} and we shall verify it for the studied example.

\begin{itemize}
\item[(H5)]
\begin{itemize}
\item[i)] There exist $C>0$ and $\lambda>1$ such that for all $z\in\bar\D$, $h\in\cB$ and $n\geq 0$,
\[
\|R(z)^n h\|_{\cB_w}\leq C|z|^n\|h\|_{\cB_w},\quad \|R(z)^n 
h\|_{\cB}\leq C\lambda^{-n}|z|^n\|h\|_{\cB}
+C|z|^n\|h\|_{\cB_w}.
\]
\item[ii)] For $z\in\bar\D\setminus\{1\}$, the spectrum of $R(z):\cB\to\cB$ does
not contain $1$.
\end{itemize}
\end{itemize}

\subsection{Proofs of Theorems~\ref{thm:mix1}, \ref{thm:mix2} and \ref{thm:mix3}}

In Section~\ref{sec-Bspace}, we verify that the map described in Section~\ref{sec:set-up}
satisfies (H1)--(H5).  In Section~\ref{sec:regvar}, we provide the set of parameters for which this map satisfies the assumptions
on $\mu(\varphi>n)$ used in Theorems~\ref{thm:mix1}, \ref{thm:mix2} and \ref{thm:mix3}. Given these facts, these 
results are an immediate consequence of previous works, as summarized below.
\\[3mm]
\begin{proofof}{Theorem~\ref{thm:mix1}} 
This follows directly from \cite[Theorem 1.1]{LT} (which is an extension of \cite[Theorem 1.1]{MT}).
\end{proofof}

\begin{proofof}{Theorem~\ref{thm:mix2}}
This follows from ~\cite[Theorem 1.2.\ i)]{LT} (an extension of \cite[Theorem 9.1]{MT})
with $2\beta$ there replaced by $\beta + \beta_*$.~\end{proofof}

\begin{proofof}{Theorem~\ref{thm:mix3}}
We note that~\cite[Theorem 1.3]{T-PTRF} holds with the present Banach space of distributions 
described in Section~\ref{sec:set-up} instead of the Banach space used in~\cite{T-PTRF}
provided  ~\cite[Lemma 2.1]{T-PTRF} is replaced by ~\cite[Lemma 3.1]{LT}.
The conclusion follows from this fact together with the argument used
at the end of~\cite[Proof of Theorem 1.2.\ i)]{LT}.~\end{proofof}

\section{Regular variation of $\mu(\varphi > n)$}\label{sec:regvar}

Let $Q := [0,\zeta_0] \times [0,\eta_0]$. We start with a harmless change of coordinates
turning $\partial Q$ into local stable and unstable leaves.

\begin{lemma}\label{lem:straight}
Let $f$ be a $C^r$ diffeomorphism\footnote{We will use this lemma for $r = \kappa+2$
and later for $r = \infty$ for the $f_j$'s in the proof of Lemma~\ref{lem:xi}.} of the form \eqref{eq:diffeo}.
 There is a $C^{\kappa+2}$ change of coordinates which
 transforms the upper and right boundary of $Q$ into the horizontal
 arc $[0, \zeta_0] \times \{ \eta_0\}$ and the vertical
 arc $\{ \zeta_0 \} \times [0, \eta_0]$ respectively, whilst 
 the diffeomorphism $f$ still has the form \eqref{eq:diffeo}.
 That is, $W^u((0, \eta_0))$ and $W^s((\zeta_0,0))$ contain locally
 a horizontal and vertical arc respectively.
\end{lemma}

\begin{proof}
 Since $f$ is $C^{\kappa+2}$, the local unstable manifold $W_{loc}^u(\eta_0)$ is a $C^{\kappa+2}$ curve
 (by an adaptation of \cite[Proposition 2.2(3)]{HY95}), 
 and it can be parametrised as $x \mapsto (x,w(x))$, say for $0 \leq x \leq \zeta_0$.
 We can also assume that $\zeta_0$ is so small that $|\eta_0-w(x)| \leq \eta_0/10$
 for $0 \leq x \leq \zeta_0$. Let $\chi:\R \to [0,1]$ be a $C^{\infty}$ bump function supported on 
 $[-\eta_0/3, \eta_0/3]$ such that $\chi(0) = 1$ and $|\chi'|_\infty \leq 3$.
 Define the diffeomorphism
 $$
 \vartheta(x,y) = (\ x\ ,\ y + \chi(y-w(x)) (\eta_0-w(x)) \ )
 $$
 and set $\tilde f = \vartheta \circ f \circ \vartheta^{-1}$. For this diffeomorphism, $W^u(\eta_0)$ is
 a horizontal line.
 Then we apply the same argument to $\tilde f$, with the roles of $x$ and $y$
 interchanged, to straighten the local stable manifold $W^s(\zeta_0)$.
\end{proof}

\subsection{Reformulation of the set-up of the neutral fixed point}\label{sec:reformulation}

Fix an even $\kappa \in \N$. Consider the differential equation $\dot z = X(z)$ for $z = (x,y)$,
defined on $Q$ and vector field $X = (X_1, X_2)$ given by \eqref{eq:vf}.
Clearly $(0,0)$ is the unique stationary point, and the time one map 
of the flow $\Phi^t$ satisfies
\begin{equation}\label{eq:time1}
\Phi^1(x,y) = \left(x(1+a_0x^\kappa+a_2y^\kappa + O( |z|^{2\kappa}))  , 
 y(1-b_0x^\kappa-b_2y^\kappa + O( |z|^{2\kappa})) \right).
\end{equation}
This can be seen by applying the Taylor expansion
$$
(\Phi^t(z))_i = z_i + X_i(z) t + \frac{t^2}{2} \left(DX(\Phi^{s_i}(z)) X(\Phi^{s_i}(z)) \right)_i
\quad \text{ for some } s_i \in [0,t], i = 1,2,
$$
applied to $t = 1$, together with the fact that $X(z) = O(|z|^{\kappa+1})$ and $|DX(\Phi^s(z))| = O(|z|^{\kappa})$.
Here we used that $|\Phi^s(z)| \leq 2|z|$, which follows because the operator
$$
\Pi: C([0,t], \R^2) \to C([0,t], \R^2), \quad (\Pi f)(s) = z+\int_0^s X(f(u)) \ du
$$
maps the ball in $(C([0,t], \R^2), \| \, \|_\infty)$ of radius $2|z|$ into itself (proved via standard methods
to show the existence and uniqueness of solutions to initial value problems, \cite[Chapter 2]{Teschl}).

The point $(0,0)$ is a neutral fixed point of saddle type:
the horizontal axis is the unstable and the vertical axis is the 
stable manifold.

\begin{remark}
One can  reduce the exponent $\kappa$ to $1$ using the change of coordinates $(\bar x,\bar y)
= (x^\kappa, y^\kappa)$. This leads to the differential equation
\begin{equation}\label{eq:coorchange}
\begin{cases}
\dot {\bar x} = \kappa \bar x(a_0\bar x+a_2 \bar y), \\
\dot {\bar y} = -\kappa \bar y(b_0\bar x+b_2 \bar y).
\end{cases}
\end{equation}
and a similar expression for \eqref{eq:time1}. 
(In fact, also when $x$ and $y$ have their
own exponents $\kappa_x, \kappa_y > 0$, the change of coordinates $(\bar x,\bar y)
= (x^{\kappa_x}, y^{\kappa_y})$ leads to the analogue of \eqref{eq:coorchange}.)
However, this change of coordinates reduces a $C^{\kappa+2}$  vector field to a
$C^{3+\frac{1}{\kappa}}$ vector field, which may be awkward to work with.
Keeping the $\kappa$ all the way through retains some clarity in this respect.
\end{remark}

\begin{remark}\label{rem:mixed}
If mixed terms, say $a_1, b_1$, are introduced, then we have no explicit
form of the first integral $L$, see \eqref{eq:lf} below.
In fact, no such $L$ may exist if the mixed terms are too large.
For example, if the vector field has the form
 $$
 \begin{cases}
\dot x = x(a_0x^2 + a_1 xy + a_2y^2), \\
\dot y = -y(b_0x^2 + b_1 xy + b_2 y^2),
\end{cases}
$$
then the origin is no longer a simple neutral saddle point if 
$(a_1 + b_1)^2 \geq 4 (a_0 + b_0)(a_2 + b_2)$.
\end{remark}

Let $\eta_1$ be such that $\Phi^{-1}(0,\eta_0) = (0,\eta_1)$. 
For initial condition $(\xi,\eta) \in \partial Q$ with 
$\eta \in [\eta_0, \eta_1]$ and $0 < \xi < \zeta_0$,
define its {\em exit time} $T(\xi, \eta)$ by
$$
\Phi^{T(\xi,\eta)}(\xi, \eta) = (\zeta_0, \omega(\eta,T)) \in \partial Q,
$$
see Figure~\ref{fig:vf} for the case $\eta = \eta_0$.
We invert this relation to obtain, for fixed $\eta$, the asymptotic behaviour
of $\xi(\eta,T)$ (and $\omega(\eta,T)$ as a byproduct) as $T \to \infty$, 
because for integer values $T = n$ and $T=n-1$,
the curves $\eta \mapsto \xi(\eta,T)$ parametrise the vertical boundaries
of the strips $\{ \varphi = n\}$.

\begin{figure}[ht]
\begin{center}
\begin{tikzpicture}[scale=1.5]
\node at (4.1,-0.1) {\small $x$};
\node at (-0.1,4.1) {\small $y$};
\draw[->] (-0.5,0)--(4.2,0);
\draw[->] (0,-0.5)--(0,4.2);
\draw[-] (0,2.8)--(2.8,2.8); \node at (0.9,3) {\small $(\xi(\eta_0, T),\eta_0)$};
\node at (0.16,2.8) {\small $\bullet$}; 
\draw[-] (2.8,0)--(2.8,2.8); \node at (3.55,0.35) {\small $(\zeta_0,\omega(\eta_0, T))$};
\node at (2.8, 0.16) {\small $\bullet$}; 
\node at (-0.16,2.8) {\small $\eta_0$};
\node at (2.8, -0.13) {\small $\zeta_0$};
\draw[->, draw=blue] (0.1,4) .. controls (0.25,0.5) and (0.5,0.25) .. (4,0.1);
\end{tikzpicture}
\caption{A solution of \eqref{eq:vf} with points $(\xi(\eta_0, T),\eta_0)$
and $(\zeta_0,\omega(\eta_0, T))$ drawn in.}
\label{fig:vf}
\end{center}
\end{figure}
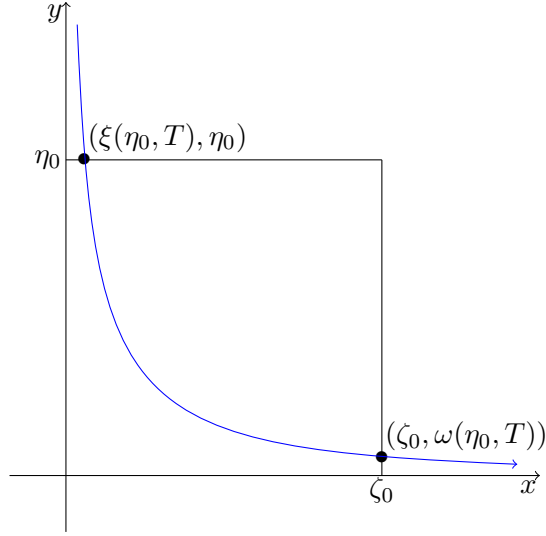 

For the vector field in \eqref{eq:vf}, i.e., without higher order terms,
very accurate estimates for $\xi(\eta,T)$ and $\omega(\eta,T)$
are given in Proposition~\ref{prop:regvar} below.
In itself, this gives no further benefit because the higher order asymptotics
are diluted by the properties of the invariant measure $\mu$,
mostly the densities $\frac{d\mu_{W^u}}{dm_{W^u}}$ conditioned to
unstable leaves $W^u$, of which we have no specific estimates.
However, Proposition~\ref{prop:regvar} sets the scene from which the general case follows
by tracking the effects of the perturbations.
\\[4mm]
Assume that $\Delta := a_2b_0 - a_0b_2 \neq 0$.
Let $u,v \in \R$ be the solutions of the linear equation
\begin{equation}\label{eq:uv}
\begin{cases}
(u+\kappa) a_0 = v b_0 \\
(v+\kappa)b_2 = u a_2
\end{cases}
\quad \text{ that is:}\quad
\begin{cases}
u = \frac{\kappa b_2}{\Delta}(a_0+b_0), \\[2mm]
v = \frac{\kappa a_0}{\Delta} (a_2+b_2).
\end{cases}
\end{equation}
Note that $u,v$ and $\Delta$ all have the same sign. Define:
\begin{equation}\label{eq:gamma}
\beta_0 := \frac{a_0+b_0}{\kappa a_0} = \frac{u+v+\kappa }{\kappa v}, \quad 
\beta_2 := \frac{a_2+b_2}{\kappa b_2} = \frac{u+v+\kappa }{\kappa u}, \quad
\begin{array}{l}c_0 = a_0 + b_0 \\ c_2 = a_2+b_2 \end{array},
\end{equation}
and note that $\frac{\beta_0}{\beta_2} = \frac{u}{v}$ and $\beta_0, \beta_2 > \frac{1}{\kappa}$ (or $=\frac{1}{\kappa}$ if 
we allow $b_0=0$ or $a_2=0$ respectively).
In the statements of the main theorems in Section~\ref{sec:mainresults}, $\beta = \beta_2$.

Recall that in our choice of coordinates
 $\{ \zeta_0 \} \times [0,\eta_0]$ is a local stable leaf.
Therefore also the curves $W_T := \xi(\eta,T),\eta)$ are local stable leaves, simply because
 $\Phi^T(W_T) \subset \{ 1 \} \times [0,\eta_0]$.
The next proposition establishes precise estimates for the parametrisation $\eta \mapsto \xi(\eta,T)$
of such local stable leaves (i.e., vertical boundaries of the strips $\{ \varphi = n\}$):

\begin{prop}\label{prop:regvar}
Consider a vector field on the $2$-torus with local form \eqref{eq:vf}
for $a_0, a_2, b_0, b_2 \geq 0$ and $\Delta \neq 0$.
There are functions $\xi_0(\eta), \omega_0(\eta), \xi_1(\eta), \omega_1(\eta) > 0$ 
independent of $T$ (with exact expressions given in the proof) 
such that 
$$
\xi(\eta, T) = \xi_0(\eta) T^{-\beta_2} \left(1 - \xi_1(\eta) T^{-1} + O(T^{-2}, T^{-\kappa \beta_2} )\right)
$$
and 
$$
\omega(\eta, T) = \omega_0(\eta) T^{-\beta_0} \left(1 - \omega_1(\eta) T^{-1} + O(T^{-2}, T^{-\kappa \beta_0} )\right).
$$
\end{prop}

\begin{remark}
 If $a_2 = 0$ (resp.\ $b_2=0$), then the term $O(T^{-\kappa \beta_2}) = O(T^{-1})$
 (resp.\ the term $O(T^{-\kappa \beta_0}) = O(T^{-1})$) and hence the term
 $\xi_1(\eta) T^{-1}$ (resp.\ $\omega_1(\eta) T^{-1})$) is no longer an 
 exact asymptotic in our estimates.
\end{remark}

\begin{remark}\label{rem:HZ}
 In \cite[Proposition 4.1]{HuZhangprepr}, it is shown that for $\kappa = 2$, $\eta = 1$ and $T = n \in \N$,
 $D_\beta n^{-1/\beta} \leq \xi(1,n) \leq  D_\alpha n^{-1/\alpha}$,
 for any choice of $\alpha, \beta$ satisfying
 $\beta < \frac{2a_2 b_2}{a_2^2 + a_2b_2 + b_2^2} < \frac{2b_2}{a_2} < \alpha$,
 and $D_\alpha, D_\beta > 0$ are independent of $n$.
 Proposition~\ref{prop:regvar} corresponds to the value $\frac{1}{\beta_2} = \frac{2b_2}{a_2+b_2}$, which lies
 in $[\frac{2a_2 b_2}{a_2^2 + a_2b_2 + b_2^2} \ , \ \frac{2b_2}{a_2}]$,
 so Proposition~\ref{prop:regvar} confirms \cite{HuZhangprepr} and makes it more precise.
\end{remark}

\begin{remark}
 The solution of \eqref{eq:vf} with initial condition $(\zeta_0,0)$ and $t \leq 0$ is
 $(x(t), y(t)) = ( (\zeta_0-\kappa a_0 t)^{-\frac1\kappa }\ , \ 0)$,
 so $x(-T) \sim T^{-\frac1\kappa }$.
 Then $(\xi(\eta_0, T),\eta_0)$ and $(x(-T), 0)$ lie on the same stable leave.
 Therefore the holonomy map $\pi: (\xi(\eta_0, T), \eta_0) \mapsto (x(-T), 0)$ 
 along the stable foliation
 is at best H\"older continuous with exponent $\frac{b_2}{a_2+b_2}$.
 This exponent tends to $1$ as $a_2 \to 0$ or $b_2 \to \infty$, as one should expect.
 Namely, if $a_2 = 0$, then the first equation of \eqref{eq:vf} is decoupled
 and can be solved directly, see also \cite{AA13} where a product type almost Anosov
 is taken as an example, and upper bounds for mixing rates are
 found using Young towers.
 For initial condition $(0, \eta_0)$,
 equation \eqref{eq:vf} his solved by $(x(t), y(t)) = (0, (b_2\kappa t + \eta_0^{-\kappa})^{-\frac1\kappa})$.
 Therefore the limit case $b_2\to\infty$ corresponds to arbitrary fast contraction along stable leaves, 
 and it is known that already exponential contraction in the stable direction produces a Lipschitz stable foliation.
\end{remark}

\noindent
The plan of action is:
\begin{enumerate}
 \item Find a first integral $L = L(x,y)$ for \eqref{eq:vf}, so that the orbits are
 confined to level sets of $L$.
 This allows also to compute the exit point $(\zeta_0,\omega) = \Phi^T(\xi,\eta_0)$
 simply because $L(\xi,\eta_0) = L(\zeta_0,\omega)$. This is done in Lemma~\ref{lem:lf} in
 Section~\ref{sec:lya}.
 \item Take the new coordinate $M = M(x,y) = y/x$, which reduces 
 \eqref{eq:vf} to a one-dimensional differential equation
\begin{equation}\label{eq:vf1}
\dot M = Z(M,L). 
\end{equation}
\item Solve \eqref{eq:vf1}, or integrate its inverse $\frac{dt}{dM} = \frac{1}{ Z(M,L) }$
to compute $T = \int_{M(\xi,\eta_0)}^{M(\zeta_0,\omega)}\frac{1}{ Z(M,L) }\ dM$
for $L = L(\xi,\eta_0)$. This and the previous step are done in Section~\ref{sec:asymp}.
\item Perturb \eqref{eq:vf} by including higher order terms, and follow the proof
of Proposition~\ref{prop:regvar} in detail to estimate the effect of this perturbation.
This is done in Section~\ref{sec:perturb} and leads to the proof of Theorem~\ref{thm:regvarmu_special}.
\item Study the transition from diffeomorphism $f$ to vector field $X$.
If $f$ is $C^\infty$ and has the form \eqref{eq:diffeo}, then it can always be written as
the time-$1$ map of a vector field of the form \eqref{eq:vfhot},
see \cite[Theorem B and consequence 1(ii) on page 36)]{DRR}. 
\item The $C^{\kappa+2}$ setting requires an extra argument
in which the vector field is approximated by $C^\infty$ vector fields with the same $\kappa+2$-jet
(i.e., the same Taylor expansion truncated at the $\lceil \kappa+2 \rceil$-nd term).
Section~\ref{sec:proof} contains this argument, and also the final step passing from length estimates
to estimates in terms of the invariant measure $\mu$.
This leads to the proof of Theorem~\ref{thm:regvarmu}.
\end{enumerate}

\subsection{First integral}\label{sec:lya}

\begin{lemma}\label{lem:lf} Recall $u,v$ and $\Delta$ from \eqref{eq:uv}.
The function
\begin{equation}\label{eq:lf}
L(x,y) = 
\begin{cases}
  x^u y^v ( \frac{a_0}{v}\ x^\kappa  + \frac{b_2}{u}\ y^\kappa ) & \text{ if } \Delta > 0;\\
  x^{-u} y^{-v} ( \frac{a_0}{v}\ x^\kappa  + \frac{b_2}{u}\ y^\kappa )^{-1} & \text{ if } \Delta < 0.
\end{cases}
\end{equation}
is a first integral of \eqref{eq:vf}.
\end{lemma}
This means that solutions of \eqref{eq:vf} are confined to level
sets $\{ L(x,y) = L_0\}$ for $L_0 \in \R$. For $L_0 = 0$, 
this level set is the union of the coordinate axes,
and (as used in Proposition~\ref{prop:perturb} below) $L$ is H\"older continuous near the axes.
For $L_0 > 0$, the level set $\{ L(x,y) = L_0\}$ roughly resembles a hyperbola in the first
quadrant, with the coordinate axes as asymptotes, see Figure~\ref{fig:vf}.
\\[3mm]
\begin{proofof}{Lemma~\ref{lem:lf}}
First assume $\Delta > 0$, so $u,v > 0$ as well.
By \eqref{eq:uv}, we can write $L(x,y)$ as
$$
L(x,y) = x^u y^v ( \frac{b_0}{u+\kappa }\ x^\kappa  + \frac{b_2}{u}\ y^\kappa )
= x^u y^v ( \frac{a_0}{v}\ x^\kappa  + \frac{a_2}{v+\kappa }\ y^\kappa ).
$$
Using these two equivalent expressions, we
compute the Lie derivative directly
\begin{eqnarray*}
\dot L = \langle \nabla L , X \rangle &=&
\left(\frac{b_0}{u+\kappa } (u+\kappa ) x^{u+\kappa-1} y^v + \frac{b_2}{u} u x^{u-1} y^{v+\kappa }\right)
x(a_0 x^\kappa +a_2 y^\kappa )\\
&& - \left(\frac{a_0}{v} x^{u+\kappa } vy^{v-1} + \frac{a_2}{v+\kappa } x^u (v+\kappa )y^{v+\kappa-1}\right)
y(b_0 x^\kappa +b_2 y^\kappa) = 0.
\end{eqnarray*}
Any function of a first integral is a first integral, in particular this holds for $1/L$.
Therefore the conclusion is immediate for $\Delta < 0$ too.
\end{proofof}

\subsection{The exact asymptotics for the vector field \eqref{eq:vf}}\label{sec:asymp}

\begin{proofof}{Proposition~\ref{prop:regvar}}
We carry out the proof for $\Delta > 0$, so 
$L(x,y) =  x^u y^v ( \frac{a_0}{v}\ x^\kappa  + \frac{b_2}{u}\ y^\kappa )$ as in Lemma~\ref{lem:lf}.
The case $\Delta < 0$ goes likewise.
Fix $\eta$ such that $(\xi(\eta,T), \eta) \in \overline{\Phi^{-1}(Q) \setminus Q}$.
For simplicity of notation, we will suppress the $\eta$ and $T$ in $\xi(\eta,T)$.
We use the variable $M = y/x$, so $y = Mx$ and differentiating gives
$\dot y = \dot M x + M \dot x$.
Recalling that $c_0 = a_0+b_0$ and $c_2 = a_2+b_2$ and inserting the values for $\dot x$ and $\dot y$ from \eqref{eq:vf},
we get
\begin{equation}\label{eq:M0}
\dot M = -M( c_0 + c_2 M^\kappa ) x^\kappa .
\end{equation}
Assume that we are in the level set 
$L(x,y) =  L(\xi,\eta) = \xi^u \eta^v(\frac{a_0}{v} \xi^\kappa  + \frac{b_2}{u}\eta^\kappa )$, then
we can solve for $x^\kappa $ in the expression
$$
\xi^u\eta^v(\frac{a_0}{v} \xi^\kappa  + \frac{b_2}{u}\eta^\kappa ) = x^u y^v (\frac{a_0}{v} x^\kappa  + \frac{b_2}{u} y^\kappa ) 
= x^{u+v+\kappa } M^v (\frac{a_0}{v} + \frac{b_2}{u} M^\kappa ).
$$
Use \eqref{eq:uv} and \eqref{eq:gamma} to obtain
$$
\frac{a_0}{v} + \frac{b_2}{u} M^\kappa  = \frac{\Delta}{\kappa c_0c_2} ( c_0+c_2M^\kappa )
\quad \text{ and }\quad 
\frac{a_0 \xi^\kappa}{v}  + \frac{b_2 \eta^\kappa }{u} = \frac{\Delta}{\kappa c_0c_2} ( c_0 \xi^\kappa +c_2 \eta^\kappa ).
$$
This gives 
$x^\kappa  = \xi^{\frac{\kappa u}{u+v+\kappa }} 
\eta^{\frac{\kappa v}{u+v+\kappa }} M^{-\frac{\kappa v}{u+v+\kappa }}
\left(\frac{c_0\xi^\kappa +c_2\eta^\kappa }{c_0+c_2M^\kappa } \right)^{\frac{\kappa }{u+v+\kappa }}$
and, combined with \eqref{eq:M0},
$$
\dot M = -M^{1-\frac{\kappa v}{u+v+\kappa }} (c_0+c_2M^\kappa )^{1-\frac{\kappa }{u+v+\kappa }} \
\xi^{\frac{\kappa u}{u+v+\kappa }} \ \eta^{\frac{\kappa v}{u+v+\kappa } } \ 
(c_0\xi^\kappa +c_2\eta^\kappa )^{\frac{\kappa }{u+v+\kappa }}.
$$
Recall
$\beta_0 = \frac{u+v+\kappa }{\kappa v}$ and $\beta_2 = \frac{u+v+\kappa }{\kappa u}$ 
from \eqref{eq:gamma} (which also gives $1-\frac{\kappa }{u+v+\kappa } = \frac{1}{\kappa \beta_0} + \frac{1}{\kappa \beta_2}$)
to simplify this to
\begin{equation}\label{eq:M3}
 \dot M = -G M^{1-\frac{1}{\beta_0}} 
 \left(c_0+c_2M^\kappa  \right)^{\frac{1}{\kappa \beta_0} + \frac{1}{\kappa \beta_2}}
\end{equation}
with
\begin{equation}\label{eq:Exi}
G = G(\xi,\eta) := \xi^{ \frac{1}{\beta_2} } \eta^{ \frac{1}{\beta_0} }
\left( c_0\xi^\kappa  + c_2\eta^\kappa  \right)^{1-\frac{1}{\kappa \beta_0}-\frac{1}{\kappa \beta_2}}.
\end{equation}
For the exit time $T \geq 0$, recall that $\xi(\eta,T)$ and $\omega(\eta,T)$ are such that 
the solution of \eqref{eq:vf} satisfies
$(x(0), y(0)) = (\xi(\eta,T), \eta)$ and $(x(T), y(T)) = (\zeta_0,\omega(\eta,T))$.
This implies $M(0) = \eta/\xi(\eta,T)$ and $M(T) = \omega(\eta,T)/\zeta_0$.
Inserting this in \eqref{eq:M3}, separating variables, and integrating we get
\begin{equation}\label{eq:M4}
\int_{\omega(\eta,T)/\zeta_0}^{\eta/\xi(\eta,T)} \frac{dM}
{  M^{1-\frac{1}{\beta_0}}  \left(c_0+c_2 M^\kappa  \right)^{ \frac{1}{\kappa \beta_0} + \frac{1}{\kappa \beta_2} } }
= G(\xi(\eta,T), \eta) T.
\end{equation}
In the rest of the proof, we will frequently suppress the dependence on $\eta$ and $T$ in $\xi(\eta,T)$
and $\omega(\eta,T)$. We know that 
$L(\xi(\eta, T), \eta) = \xi^u \eta^v(\frac{a_0}{v} \xi^\kappa  + \frac{b_2}{u}\eta^\kappa ) = 
\zeta_0^u \omega^v (\frac{a_0}{v} \zeta_0^\kappa + \frac{b_2}{u} \omega^\kappa )
= L(\eta,\omega(\eta,T))$,
which gives
\begin{equation}\label{eq:xieta}
\xi^u \eta^v (c_0\xi^\kappa +c_2\eta^\kappa ) = \zeta_0^u \omega^v(c_0\zeta_0^\kappa+c_2\omega^\kappa ).
\end{equation}
From their definition, $\xi(\eta,T)$ and $\omega(\eta, T)$ are clearly decreasing in $T$,
so their $T$-derivatives $\xi'(\eta,T), \omega'(\eta,T) \leq 0$. 
Since $c_0, c_2 > 0$ (otherwise $\Delta = 0$), the integrand of \eqref{eq:M4} is $O(M^{\frac{1}{\beta_0}-1})$ as $M \to 0$ and 
$O(M^{-\frac{1}{\beta_2}-1})$ as $M \to \infty$. Hence the integral is increasing and bounded in $T$.
But this means that $C_{\xi(\eta,T)} T$ is increasing and bounded as well.
Let $g(\eta,T) = \xi(\eta,T) T^{\beta_2}$. Since
$$
G(\xi(\eta,T), \eta) T  = g(\eta,T)^{\frac{1}{\beta_2}} \eta^{\frac{1}{\beta_0}} (c_0 g(\eta,T)^\kappa  T^{-\kappa  \beta_2}
+ c_2 \eta^\kappa )^{1-\frac{1}{\kappa \beta_0} - \frac{1}{\kappa \beta_2}},
$$
and $1-\frac{1}{\kappa \beta_0} - \frac{1}{\kappa \beta_2} > 0$,
we find that $g(\eta,T)$ converges
\footnote{For the symmetric statement on $\omega(\eta,T)$, define $\hat g(\eta,T) = \omega(\eta, T) T^{\beta_0}$.
Then $\lim_{T \to \infty} \hat g(\eta,T) = \lim_{T \to \infty} g(\eta,T)^{ \beta_0/\beta_2 } \eta^{1+\kappa /v} 
\zeta_0^{-b_0/a_0} (\frac{c_2}{c_0})^{1/v}$.}:
\begin{equation}\label{eq:xi0}
\xi_0(\eta) := \lim_{T\to\infty} g(\eta,T) = c_2^{-\frac1u} \eta^{- \frac{a_2}{b_2} } \left( \int_0^\infty \frac{ dM }
{  M^{1-\frac{1}{\beta_0}}  \left( c_0 + c_2 M^\kappa  \right)^{ \frac{1}{\kappa \beta_0} + \frac{1}{\kappa \beta_2} } }
\right)^{\beta_2} ,
\end{equation}
where we have used $-\beta_2(1-\frac{1}{\kappa \beta_0} - \frac{1}{\kappa \beta_2}) = -\frac{\kappa \beta_2}{u+v+\kappa } = -\frac{1}{u}$
for the exponent of $c_2$, and $\frac{\kappa }{u} + \frac{\beta_2}{\beta_0} = \frac{v+\kappa }{u} = \frac{a_2}{b_2}$ 
for the exponent of $\eta$.

We continue the proof to get higher asymptotics.
Differentiating \eqref{eq:M4} w.r.t.\ $T$ gives
\begin{equation}\label{eq:Mprime}
-\frac{ \eta^{\frac{1}{\beta_0}} \xi^{\frac{1}{\beta_2}-1} \xi'}{(c_0 \xi^\kappa  + c_2\eta^\kappa )^{\frac{1}{\kappa \beta_0} + \frac{1}{\kappa \beta_2}}}
-\frac{\zeta_0^{\frac{1}{\beta_2}} \omega^{\frac{1}{\beta_0}-1} \omega'}{(c_0\zeta_0^\kappa + c_2 \omega^\kappa )^{\frac{1}{\kappa \beta_0} + \frac{1}{\kappa \beta_2}}}
=  \frac{\partial G(\xi,\eta)}{\partial\xi} T \xi'  + G(\xi,\eta),
\end{equation}
where (by differentiating \eqref{eq:Exi})
$$
\frac{\partial G(\xi,\eta)}{\partial\xi}  = \kappa  (b_0\xi^\kappa +b_2\eta^\kappa ) \xi^{\frac{1}{\beta_2}-1} \eta^{\frac{1}{\beta_0}}
 (c_0\xi^\kappa  + c_2 \eta^\kappa )^{-\frac{1}{\kappa \beta_0} - \frac{1}{\kappa \beta_2}}.
$$
Combined with \eqref{eq:xieta} and \eqref{eq:Mprime} this gives
\begin{equation}\label{eq:xi2}
-\eta^{\frac{1}{\beta_0}} \xi' - \zeta_0^{\frac{1}{\beta_2}}
\left(\frac{\zeta_0^u \omega^v}{\xi^u \eta^v} \right)^{\frac{1}{\kappa \beta_0} + \frac{1}{\kappa \beta_2}} 
\frac{\omega^{\frac{1}{\beta_0}-1}}{\xi^{\frac{1}{\beta_2}-1}}  \omega' =
\kappa (b_0 \xi^\kappa  + b_2\eta^\kappa ) T \eta^{\frac{1}{\beta_0}} \xi' + \eta^{\frac{1}{\beta_0}} \xi (c_0 \xi^\kappa +c_2 \eta^\kappa ).
\end{equation}
Because $\frac{1}{\kappa \beta_0} + \frac{1}{\kappa \beta_2} - 1 = -\frac{\kappa }{u+v+\kappa }$, using
\eqref{eq:gamma} and dividing by $\eta^{\frac{1}{\beta_0}}$, we can simplify \eqref{eq:xi2} to 
\begin{equation}\label{eq:xi2a}
- \xi' - \frac{\zeta_0^u}{\eta^v} \frac{\omega^{v-1}}{\xi^{u-1}} \omega' =
\kappa (b_0 \xi^\kappa  + b_2\eta^\kappa ) T \xi' + \xi (c_0 \xi^\kappa +c_2 \eta^\kappa ).
\end{equation}
Taking the derivative of \eqref{eq:xieta} w.r.t.\ $T$  and multiplying with $\Delta/(\kappa c_0c_2)$ gives
$$
( b_0 \xi^\kappa + b_2\eta^\kappa ) \eta^v \xi^{u-1} \xi' = (a_0\zeta_0^\kappa + a_2 \omega^\kappa )\zeta_0^u \omega^{v-1} \omega'.
$$
Hence, we can rewrite \eqref{eq:xi2a} as
$$
-(1 +\frac{ b_0 \xi^\kappa + b_2\eta^\kappa  }{a_0 \zeta_0^\kappa+ a_2 \omega^\kappa }) \xi' =
\kappa (b_0 \xi^\kappa  + b_2 \eta^\kappa ) T  \xi' +  \xi (c_0 \xi^\kappa +c_2 \eta^\kappa ).
$$
We insert $\xi' = g'(T) T^{-\beta_2} - \beta_2 g(T) T^{-(1+\beta_2)}$ and multiply with $T^{\beta_2}$,
which leads to
\begin{equation*}
-(1+\frac{b_0 \xi^\kappa +b_2\eta^\kappa }{a_0\zeta_0^\kappa +a_2\omega^\kappa }) (g'(T) - \beta_2 g(T) T^{-1})
 = \kappa (b_0\xi^\kappa  + b_2\eta^\kappa ) g'(T) \ T 
-\frac{\Delta}{b_2} g(T)^{\kappa+1} T^{-\kappa \beta_2}.
\end{equation*}
Since $\xi = O(T^{-\beta_2})$ and $\omega = O(T^{-\beta_0})$, we can write this differential equation as
$$
\frac{g'}{g} = \frac1{T^2} \frac{\beta_2}{\kappa} 
\frac{\frac{a_0\zeta_0^\kappa +b_2 \eta^\kappa + O(T^{-\kappa \beta_2})}{a_0\zeta_0^\kappa +O(T^{-\kappa \beta_0})}
-\frac{\Delta}{b_2} g(T)^\kappa T^{-\frac{a_2}{b_2}} }
{b_2 \eta^\kappa + O(T^{-\kappa \beta_2}) + O(T^{-1})}.
$$
Keeping the leading
terms only (where we use that $\kappa \beta_2, \kappa \beta_0 > 1$), we get the differential equation
$$
\frac{g'}{g} = (\xi_1(\eta) + O(\max\{T^{-1}, T^{-\frac{a_2}{b_2}}\} ) ) \frac{1}{T^2} \quad \text{ for }
\xi_1 = \xi_1(\eta) := \frac{\beta_2}{\kappa} \left( \frac{1}{a_0\zeta_0^\kappa}+ \frac{1}{b_2\eta^\kappa } \right).
$$
Using the limit boundary value $\xi_0 = \xi_0(\eta) = \lim_{T \to \infty} g(\eta,T)$, we find the solution
$$
g(\eta,T) = \xi_0 e^{-(\xi_1 + O(\max\{T^{-1}, T^{-\frac{a_2}{b_2}}\} ) )T^{-1} } 
= \xi_0( 1 - \xi_1 T^{-1}  + O(\max\{T^{-2}, T^{-\kappa \beta_2}\} ) )
$$
as required. The analogous asymptotics for $\omega$ and the constants $\omega_0$ and $\omega_1$
can be derived by changing the time direction and the roles 
$(a_0, a_2) \leftrightarrow (b_2, b_0)$, and also by the relation $\xi^u \eta^{v+\kappa }c_2 \sim \zeta_0^{u+\kappa} \omega^v c_0$
from \eqref{eq:xieta}:
$$
\omega_0(\eta) := c_0^{-1/v} \left( \int_0^\infty \frac{ dM }
{  M^{1-\frac{1}{\beta_2}}  \left( c_0 M^\kappa  + c_2 \right)^{ \frac{1}{\kappa \beta_0} + \frac{1}{\kappa \beta_2} } }
\right)^{\beta_0}
\ \text{ and } \ \omega_1(\eta) := \frac{\beta_0}{\kappa } \left(\frac{1}{b_2\zeta_0^\kappa}+ \frac{1}{a_0\eta^\kappa }\right).
$$
This concludes the proof.
\end{proofof}

Recall that the strip $\{\varphi = n\}$ is bounded by the unstable curve
$W^u = \{ y = \eta_0\}$ forming the upper boundary of $Q$,
its preimage $\Phi^{-1}(\{ y = \eta_0 \})$ (see $f^{-1}(W^u)$ in Figure~\ref{fig:leaves})
and the stable curves
through the points $\xi(\eta_0,n)$ and $\xi(\eta_0,n-1)$.
Let $\eta_1$ be such that $\Phi^{-1}(\{ y = \eta_0 \})$ intersect the vertical axis at $(0,\eta_1)$.
Proposition~\ref{prop:regvar} gave us the estimates for $\xi(\eta,n)$ for $\eta_0 \leq \eta \leq \eta_1$.
Therefore, the remaining step is to pass from the Lebesgue measure $m(\varphi = n)$ 
to the SRB-measure $\mu(\varphi = n)$.
\medskip

\begin{proofof}{Theorem~\ref{thm:regvarmu_special}}
Every local unstable leaf $W^u$ intersects the local stable leaf $W^s$ of $p$
in a unique point $(0,y)$, so we parametrise these local unstable leaves as
$W^u(y)$. Also the conditional measure $\mu^u_{W^u(y)}$ is absolutely continuous
w.r.t.\ Lebesgue, so we can write $d\mu^u_{W^u(y)} = h(x,y) dm^u_{W^u(y)}$,
and in fact $h(x,y)$ is $C^{\kappa+1}$ times differentiable in $x$
(see \cite[Proposition 3.1]{HY95}).
We decompose $d\mu = d\mu^u_{W^u(y)} d\mu^s$ and obtain
 \begin{eqnarray*}
  \mu(\varphi > n) &=& \int_{\eta_0}^{\eta_1}  \int_{W^u(y)} 1_{\{\varphi > n\}} 
  d\mu^u_{W^u(y)} \ d\mu^s(y)\\
  &=& \int_{\eta_0}^{\eta_1}  \int_{W^u(y)} 1_{\{\varphi > n\}}  h_{W^u(y)}(x,y) \ dm^u_{W^u(y)}(x) \ d\mu^s(y).
 \end{eqnarray*}
Use the Taylor expansion to approximate the inner integral.
Define $h_j(y) := \frac{\partial^j}{\partial x^j} h_{W^u(y)} (0,y)$,
so that $h_{W^u(y)}(x,y) = \sum_{j=0}^{\kappa-1} \frac{1}{j!} h_j(y) x^j + O(x^\kappa)$.
Integration over $0 \leq x \leq \xi(y,n) = \xi_0(y) n^{-\beta_2}(1-\xi_1(y)n^{-1} + 
O(\max\{ n^{-2}, n^{-\kappa \beta_2} \}) )$
gives
\begin{align*}
 \int_0^{\xi(y,n)} & h_{W^u(y)}(x,y) \, dm^u_{W^u(y)}(x)  
 =\sum_{j=1}^\kappa \frac{1}{j!} h_{j-1}(y) \xi_0^j(y) n^{-j\beta_2} \\
 & -\sum_{j=1}^\kappa \frac{1}{(j-1)!} h_{j-1}(y) \xi_0^j(y)\xi_1(y) n^{-(j\beta_2+1)}
 + \ O(\max\{ n^{-(2+\beta_2)}, n^{-(\kappa+1) \beta_2} \} )).
\end{align*}
Next we integrate over $y \in (\eta_0, \eta_1)$ and obtain
$$
\mu(\varphi > n) =
\sum_{j=1}^\kappa H_j n^{-j\beta_2} \\
-\sum_{j=1}^\kappa \hat H_j n^{-(j\beta_2+1)} + \ O(\max\{ n^{-(2+\beta_2)}, n^{-(\kappa+1) \beta_2}\})),
$$
for $H_j = \frac{1}{j!}\int_{\eta_0}^{\eta_1} h_{j-1}(y) \xi_0^j(y)\, d\mu^s(y)$
and $\hat H_j = \frac{1}{(j-1)!}\int_{\eta_0}^{\eta_1} h_{j-1}(y) \xi_0^j(y) \xi_1(y)\, d\mu^s(y)$.
This completes the proof.
\end{proofof}

\subsection{The effect of small perturbations}\label{sec:perturb}

To prove that the regular variation established in Proposition~\ref{prop:regvar}
is robust under perturbations of the vector field, we perturb $X$ from \eqref{eq:vf}
to obtain
\begin{equation}\label{eq:perturb}
\tilde X = \begin{pmatrix} \tilde X_1 \\ \tilde X_2 \end{pmatrix}
= \begin{pmatrix}
   x(a_0 x^\kappa + a_2 y^\kappa + O(|(x,y)|^{\kappa+1})) \\
  -y(b_0 x^\kappa + b_2 y^\kappa + O(|(x,y)|^{\kappa+1}))
  \end{pmatrix},
\end{equation}
so that $\tilde X - X = O(|(x,y)|^{\kappa+1})$.
The quantity $\xi(\eta, T)$ then becomes $\tilde \xi(\eta, T)$ and the goal is to show
that $\tilde \xi(\eta, T)$ is still regularly varying.

\begin{prop}\label{prop:perturb}
Consider a $C^{\kappa+1}$ vector field of local form \eqref{eq:perturb} with $a_0, a_2, b_0, b_2 \geq 0$
and $\Delta \neq 0$.
Recall that $\beta_2 = \frac{a_2+b_2}{\kappa b_2}$ and 
$\beta_* = \frac{1}{\kappa}
\min\left\{ 1,\frac{a_2}{b_2}, \frac{b_0}{a_0}\right \}$.  
Then the asymptotics of the perturbed version of $\xi(\eta,T)$ is
$$
\tilde \xi(\eta,\tilde T) = \xi_0(\eta) \tilde T^{-\beta_2} (1 + O(\tilde T^{-\beta_*}, T^{-\frac{1}{\kappa}} \log T)).
$$ 
as $T \to \infty$, and $\xi_0(\eta)$ is as in Proposition~\ref{prop:regvar}.
\end{prop}

\begin{proof}
As before, let $\xi = \xi( \eta,T)$ be such that for the unperturbed flow, 
$\Phi^T (\xi, \eta) = (\zeta_0, \omega(\eta,T))$.
Proposition~\ref{prop:regvar} gives the asymptotics of $\xi(\eta,T)$ as $T \to \infty$.
At the same time, under the perturbed flow associated to \eqref{eq:perturb},
$\Phi^{\tilde T} (\xi, \eta) = (\zeta_0, \tilde \omega(\eta,\tilde T))$ for some $\tilde T$.
Therefore we can write $\xi(\eta,T)  = \tilde \xi(\eta, \tilde T)$,
and once we estimated $\tilde T$ as function of $T$, we can express $\tilde \xi(\eta, \tilde T)$
explicitly as function of $\tilde T$. We follow the argument of the proof of Proposition~\ref{prop:regvar}, 
keeping track of the effect of the perturbations. \\[3mm]
{\bf The perturbed first integral:} To start, we construct a first integral 
 $\tilde L$ on $Q = [0,\zeta_0] \times [0,\eta_0]$ by defining 
 $$
 \tilde L(\tilde \Phi^t(\delta, \delta)) = L(\delta,\delta) = 
 \begin{cases}
  \delta^{u+v+\kappa}(\frac{a_0}{v} + \frac{b_2}{u})  & \text{ if } \Delta > 0,\\
  \delta^{-(u+v+\kappa)}(\frac{a_0}{v} + \frac{b_2}{u})^{-1}  & \text{ if } \Delta < 0,
 \end{cases}
 $$ 
 for $0 < \delta \leq \min\{\zeta_0,  \eta_0\}$ and $t \in \R$.
 (We continue the argument for the case $\Delta > 0$; the other case goes analogously.)
 
 By construction, $\tilde L$ is constant on integral curves
 of $\dot z = \tilde X(z)$.
 Because $\tilde X$ is $C^{\kappa+1}$, the integral curves are $C^{\kappa+1}$ curves,
 and form a $C^{\kappa+1}$ foliation of $P_0$, see e.g.\ \cite[Theorem 2.10]{Teschl}.
 Note that the coordinate axes consist
 of the stationary point $(0,0)$ and its stable and unstable manifold;
 we put $\tilde L(x,0) = \tilde L(0,y) = 0$.
 Then $\tilde L$ is continuous on $Q$ and $C^{\kappa+1}$ on the interior of $Q$.
 
 Now we compare $\tilde L$ with $L$ on a small neighbourhood $U$ of 
 $\Phi^{-1}(Q) \cup Q \cup \Phi^1(Q)$.
 Take $y_0 = \eta_0$ and $x_0 = x_0(\delta)$ such that the integral curve
 of $\dot z = X(z)$ through $z_0 := (x_0, y_0)$  intersects the diagonal at $(\delta, \delta)$.
 Then the integral curve
 of $\dot z = \tilde X(z)$ through $z_0$ intersects the diagonal at $(\tilde \delta, \tilde \delta)$
 for some $\tilde \delta = \tilde \delta(\delta)$, see Figure~\ref{fig:deltas}.
%
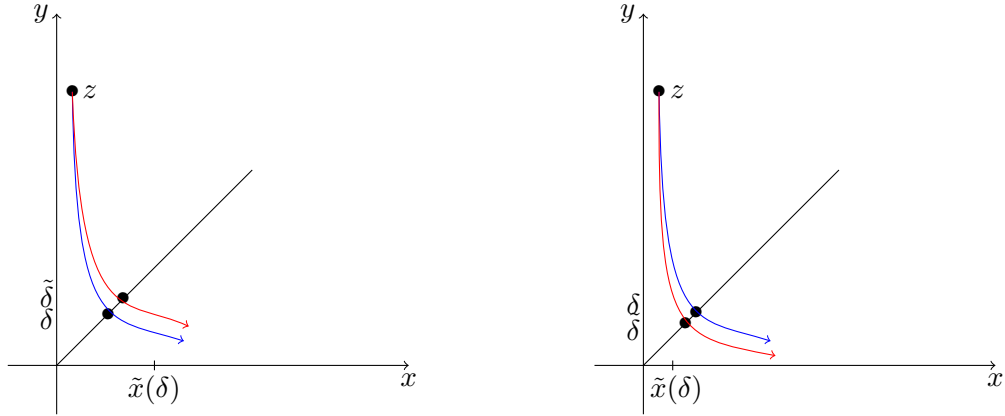
\begin{figure}[ht]
\begin{center}
\begin{tikzpicture}[scale=1.3]
\node at (3.6,-0.15) {\small $x$};\node at (-0.16,3.6) {\small $y$};
\draw[->] (-0.5,0)--(3.6,0);
\draw[->] (0,-0.5)--(0,3.6);
\draw[-] (0,0)--(2,2);
\node at (0.16, 2.8) {\small $\bullet$}; \node at (0.34,2.8) {\small $z$};
\node at (0.525, 0.525) {\small $\bullet$}; \node at (-0.1,0.48) {\small $\delta$};
\node at (0.68, 0.68) {\small $\bullet$}; \node at (-0.1,0.74) {\small $\tilde\delta$};
\draw[-] (1,-0.05)--(1,0.05); \node at (1,-0.25) {\small $\tilde x(\delta)$};
\draw[->, draw=blue] (0.16,2.8) .. controls (0.2,0.3) and (0.5,0.5) .. (1.3,0.25);
\draw[->, draw=red] (0.16,2.8) .. controls (0.25,0.3) and (0.7,0.7) .. (1.35,0.4);
\node at (9.6,-0.15) {\small $x$};\node at (5.84,3.6) {\small $y$};
\draw[->] (5.5,0)--(9.6,0);
\draw[->] (6,-0.5)--(6,3.6);
\draw[-] (6,0)--(8,2);
\node at (6.16, 2.8) {\small $\bullet$}; \node at (6.35,2.8) {\small $z$};
\node at (6.54, 0.54) {\small $\bullet$}; \node at (5.9,0.63) {\small $\delta$};
\node at (6.43, 0.43) {\small $\bullet$}; \node at (5.9,0.4) {\small $\tilde\delta$};
\draw[-] (6.3,-0.05)--(6.3,0.05); \node at (6.34,-0.25) {\small $\tilde x(\delta)$};
%
\draw[->, draw=blue] (6.16,2.8) .. controls (6.2,0.3) and (6.5,0.5) .. (7.3,0.25);
\draw[->, draw=red] (6.16,2.8) .. controls (6.15,0.3) and (6.3,0.3) .. (7.35,0.1);
\end{tikzpicture}
\caption{Solutions of \eqref{eq:vfq1} and \eqref{eq:vfq2}, starting from the same point
$z = (x, y)$. The left and right panel refer to the cases $\tilde\delta > \delta$
and $\tilde \delta < \delta$ respectively.}
\label{fig:deltas}
\end{center}
\end{figure} 

Therefore
\begin{equation}\label{eq:LL}
\tilde L(z) \ = \ \tilde L(\tilde \delta, \tilde\delta)\ =\  L(\delta,\delta)
\left(\frac{\tilde\delta}{\delta}\right)^{u+v+\kappa}
\ =  \ L(z) \left( \frac{\tilde\delta}{\delta}\right)^{u+v+\kappa} \!\!\!\!\!\!\! \!\!\!\!\!\!\!\!\!\!.
\end{equation}
\noindent
{\bf Estimating $\tilde \delta/\delta$:} Parametrise the integral curve of $X$ through $z_0$ as
$(x(y), y)$ for $\min\{ \delta, \tilde \delta\} \leq y \leq y_0$.
(So $x \leq y$; the case $y \leq x$ can be dealt with by switching the roles of $x$ and $y$.)
Then by \eqref{eq:vf}:
\begin{equation}\label{eq:vfq1}
 x'(y) = -\frac{x(a_0 x^{\kappa} + a_2 y^{\kappa})}{y(b_0 x^{\kappa} + b_2 y^{\kappa})}.
\end{equation}
For the perturbed vector field \eqref{eq:perturb} we parametrise the integral curve of through $z_0$ as
$(\tilde x(y), y)$ and we have the analogue of \eqref{eq:vfq1}:
\begin{equation}\label{eq:vfq2}
 \tilde x'(y) = -\frac{\tilde x(a_0 \tilde x^{\kappa} + a_2 y^{\kappa} 
 + \sum_{j=0}^{\kappa+1} \hat a_j \tilde x^j y^{\kappa+1-j} + o(|(\tilde x,y)|^{\kappa+1}))}
 {y(b_0 \tilde x^{\kappa} + b_2 y^{\kappa} + \sum_{j=0}^{\kappa+1} \hat b_j \tilde x^j y^{\kappa+1-j}
 + o(|(\tilde x,y)|^{\kappa+1}))}.
\end{equation}
Sincee $x \leq y$, the $O$-terms can be written as $O(y^{\kappa+1})$.
Combining \eqref{eq:vfq1} and \eqref{eq:vfq2} we obtain
$$
\tilde x'(y)  =  -\frac{\tilde x(a_0 \tilde x^{\kappa} + a_2 y^{\kappa})}{y(b_0\tilde x^{\kappa} + b_2 y^{\kappa})}
(1+ q_0 \tilde x + q_2 y +  o(|(\tilde x,y)|))
= x'(y)(1+q_0 x + q_2 y +  o(|(\tilde x,y)|)) ).
$$
We will neglect the term $o(|(\tilde x,y)|)$ because they can be absorbed
in the big-$O$ terms at the end of the estimate.
Integration over $[\delta,y_0]$ gives
$$
\tilde x(y_0) - \tilde x(\delta) = x(y_0) - x(\delta)  + q_0 \int_\delta^{y_0} x'(y) x(y)\ dy
+ q_2 \int_\delta^{y_0} x'(y) y\ dy.
$$
Since $\tilde x(y_0) = x(y_0) = x_0$ and $x(\delta) = \delta$, this simplifies to
\begin{eqnarray}\label{eq:dtilded}
\tilde x(\delta) - \delta &=&  -  \frac{q_0}{2} \int_\delta^{y_0} (x^2(y))' \ dy -
q_2 \int_\delta^{y_0} (x(y) y)' \ dy + q_2 \int_\delta^{y_0} x(y) \ dy  \nonumber \\
&=& \frac{q_0}{2}(\delta^2 - x_0^2) + q_2 \left( \delta^2 - x_0 y_0 + \int_\delta^{y_0} x(y) \ dy \right).
\end{eqnarray}
We solve for $x$ from 
$x^u y^v (\frac{a_0}{v} x^\kappa + \frac{b_2}{u}y^\kappa) = L(x, y) = L(\delta, \delta) 
= \delta^{u+v+\kappa}(\frac{a_0}{v}+ \frac{b_2}{u})$:
\begin{eqnarray}\label{eq:x(y)}
 x = x(y) &=& \delta^{\frac{u+v+\kappa}{u}} y^{-\frac{v+\kappa}{u}} (1+\frac{u a_0}{v b_2})^{\frac1u}
 (1+ \frac{c_0}{c_2} \frac{x^{\kappa}}{y^{\kappa}})^{-\frac1u} \nonumber \\
  &=& \underbrace{(c_2+c_0)^{\frac1u} 
  (c_2 + c_0 \frac{x^{\kappa}}{y^{\kappa}})^{-\frac1u}}_{U(y)}   \delta^{1+\frac{a_2}{b_2}} y^{-\frac{a_2}{b_2}}.
\end{eqnarray}
In particular,
$$
x_0 = x(y_0) = \delta^{1+\frac{a_2}{b_2}}(1+\frac{c_0}{c_2})^{\frac1u} y_0^{-\frac{a_2}{b_2}} 
  (1+O(\delta^{\kappa(1+\frac{a_2}{b_2})})). 
$$
Combine the first two factors of \eqref{eq:x(y)} to
$U(y) :=  (c_2+c_0)^{\frac{1}{u}} (c_2 + c_0 \frac{x^{\kappa}}{y^{\kappa}})^{-\frac1u} 
\in [1, (1+\frac{c_0}{c_2})^{\frac{1}{u}}]$.
Note that $\lim_{y \to \delta} U(y) = 1$, and $U(y)$ is differentiable.
Using \eqref{eq:vfq1} and \eqref{eq:x(y)} we compute the derivative
$$
U'(y) = \frac{\Delta c_0}{b_2}\ \frac{U(y)}{b_2+b_0 (\frac{x}{y})^\kappa}\ \frac1y\ (\frac{x}{y})^\kappa
= \frac{\Delta c_0}{b_2}\ \delta^{\kappa(1+\frac{a_2}{b_2})}\  \frac{U(y)^{\kappa+1}}{b_2+b_0 (\frac{x}{y})^\kappa} 
\ y^{-\kappa(1+\frac{a_2}{b_2})-1}.
$$
Next we integrate by parts (assuming first that $\frac{a_2}{b_2} \neq 1$):
\begin{eqnarray*}
\int_\delta^{y_0} x(y) \ dy &=& \delta^{1+\frac{a_2}{b_2}} \int_\delta^{y_0} U(y) y^{-\frac{a_2}{b_2}} dy
= \frac{b_2}{b_2-a_2}
\left( U(y_0) \delta^{1+\frac{a_2}{b_2}} y_0^{1-\frac{a_2}{b_2}} - \delta^2 \right) \\
&&  \qquad \underbrace{- \frac{\Delta c_0}{b_2-a_2}\ \delta^{(\kappa+1)(1+\frac{a_2}{b_2})}\  
\int_\delta^{y_0} \frac{U(y)^{\kappa+1}}{b_2+b_0 (\frac{x}{y})^\kappa} \  
y^{1-(\kappa+1)(1+\frac{a_2}{b_2})} dy}_I.
\end{eqnarray*}
Since $\frac{1}{b_2+b_0} \leq  \frac{U(y)^{\kappa+1}}{b_2+b_0 (\frac{x}{y})^\kappa}
\leq \frac{1}{b_2}(1+\frac{c_0}{c_2})^{\frac{\kappa+1}{u}}$
and $\frac{U(y)^{\kappa+1}}{b_2+b_0 (\frac{x}{y})^\kappa} \to \frac{1}{b_2+b_0}$ as $y \to \delta$,
there are constants $\hat C_1, \hat C_2 \in \R$ 
such that the final term in the above expression is
$$
I = \hat C_1 \delta^2 + \hat C_2 \delta^{(\kappa+1)(1+\frac{a_2}{b_2})}
y_0^{2-(\kappa+1)(1+\frac{a_2}{b_2})} + O(\delta^3).
$$
For the case $\frac{a_2}{b_2} = 1$, a similar computation gives
$$
\int_\delta^{y_0} x(y) \ dy = \hat C_3 \delta^2 \log \delta + \hat C_4 \delta^2 \log y_0
+ \hat C_5 \delta^{2(\kappa+1)} y_0^{-2\kappa} \log y_0
+  \hat C_6 \delta^{2(\kappa+1)} y_0^{-2\kappa} + O(\delta^3 \log \delta),
$$
for some generically nonzero $\hat C_3, \hat C_4, \hat C_5, \hat C_6 \in \R$.

By \eqref{eq:vfq2}, the derivative $\tilde x'(\delta) = \frac{a_0+a_2}{b_0+b_2} + O(\delta)$.
Since $\tilde \delta$ lies between $\delta$ and $\tilde x(\delta)$ (see Figure~\ref{fig:deltas}), 
we have 
\begin{equation}\label{eq:xdelta}
|\tilde x(\delta) - \delta| = |\tilde x(\delta) - \tilde \delta| + |\tilde \delta - \delta|
= \left(1+ \frac{a_0+a_2}{b_0+b_2} + O(\delta)\right) |\tilde \delta - \delta|
= \frac{c_0+c_2 + O(\delta) }{b_0+b_2}|\tilde \delta - \delta|.
\end{equation}
Later in the proof we need the quantity
$$
\psi(\delta) := \left(\frac{\tilde\delta}{\delta}\right)^{u+v+\kappa} - 1 = 
(u+v+\kappa) \frac{\tilde \delta - \delta}{\delta} + O\left(\frac{|\tilde \delta - \delta|^2}{\delta^2}\right).
$$
Writing $|\tilde \delta-\delta|$ in terms of $|\tilde x(\delta)-\delta|$ using
\eqref{eq:xdelta}, and
combining with the above estimates for $|\tilde x(\delta)-\delta|$, we find 
\begin{eqnarray}\label{eq:psidelta}
\psi(\delta) &=& C_1 \delta + C_2 \delta^{\frac{a_2}{b_2}} y_0^{1-\frac{a_2}{b_2}} 
 + C_3 \delta^{1+\frac{a_2}{b_2}} y_0^{-\frac{a_2}{b_2}} 
 + C_4 \delta^{1+\frac{2a_2}{b_2}} y_0^{-\frac{2a_2}{b_2}} \nonumber  \\
&& +\ C_{\log} \delta \log \delta + O(\delta^2, \delta^{\frac{2a_2}{b_2}} )
\end{eqnarray}
for (generically nonzero) constants $C_1, C_2, C_3, C_4 \in \R$
and $ C_{\log}$ is only nonzero if $\frac{a_2}{b_2} = 1$.

For the region $\{ x \geq y \}$ (containing the point $(x_1, y_1) := (\zeta_0, \tilde \omega(\eta,\tilde T))$)
we reverse the roles  $a_2,b_2,x_0, y_0 \leftrightarrow b_0, a_0, y_1, x_1$.
This gives
\begin{eqnarray*}
\psi(\delta) &=& \hat C_1 \delta + \hat C_2 \delta^{\frac{b_0}{a_0}} x_1^{1-\frac{b_0}{a_0}} 
 + \hat C_3 \delta^{1+\frac{b_0}{a_0}} x_1^{-\frac{b_0}{a_0}} 
 + \hat C_4 \delta^{1+\frac{2b_0}{a_0}} x_1^{-\frac{2b_0}{a_0}} \nonumber  \\
&& +\ \hat C_{\log} \delta \log \delta + O(\delta^2, \delta^{\frac{2b_0}{a_0}} ),
\end{eqnarray*}
for (generically nonzero) constants $\hat C_1, \hat C_2, \hat C_3, \hat C_4 \in \R$
and $\hat C_{\log}$ is only nonzero if $\frac{b_0}{a_0} = 1$.
Combining with \eqref{eq:psidelta} gives
\begin{equation}\label{eq:psidelta2}
 \psi(\delta) =
\begin{cases}
O(\delta \log 1/\delta) & \text{ if }  \min\{\frac{a_2}{b_2}, \frac{b_0}{a_0}\} = 1,\\
 O(\delta^{a_*}) & \text{ otherwise, with } a_* = \min\{ 1, \frac{a_2}{b_2}, \frac{b_0}{a_0}\}.
\end{cases}
\end{equation}

\noindent
{\bf Estimate of $\tilde T$:} Now let $z_0 = (x_0, y_0) = (\xi(\eta, T), \eta) = (\tilde\xi(\eta, \tilde T), \eta)$ be the point such that
$\Phi^T(z_0) = (\zeta_0, \omega(\eta,T))$ under the {\bf unperturbed}
flow and $\tilde \Phi^{\tilde T}(z_0) = (\zeta_0, \tilde \omega(\eta,\tilde T))$ under the {\bf perturbed}
flow.
We estimate $\tilde T$ in terms of $T$.

Combining the estimate for $\xi(\eta, T)$ from Proposition~\ref{prop:regvar}
with $L(\delta,\delta) = L(\xi(\eta,T), \eta)$, we can find the relation
between $\delta$ and $T$:
\begin{equation}\label{eq:deltaT}
 \delta = \delta_0 T^{-\frac{1}{\kappa}} (1 +\frac{\xi_1}{\kappa\beta_2} T^{-1}
 + O(T^{-2}, T^{-\kappa \beta_2})),
\end{equation}
for $\delta_0 = \xi_0^{\frac{1}{\kappa \beta_2}} \eta^{1-\frac{1}{\kappa \beta_2}}
(\frac{c_2}{c_0+c_2})^{\frac{1}{u+v+\kappa}}$.

For $M = y/x$, computations analogous to \eqref{eq:M0} show that there is
$\Psi = \Psi(x,M) = O(1+M^{\kappa+1})$ such that
$$
\dot M = - M (c_0 + c_2 M^\kappa + x \Psi) x^\kappa.
$$
For every $(x,y) = (x,xM)$ on the $\tilde \Phi$-trajectory of $z_0$ (i.e., level set of $\tilde L$),
we have
$$
\xi^u \eta^v (\frac{a_0}{v} \xi^\kappa + \frac{b_2}{u} \eta^\kappa)(1+\psi(\xi,\eta)) = 
x^{u+v+\kappa} M^\kappa (\frac{a_0}{v} + \frac{b_2}{u} M^\kappa) (1+\psi(x,xM)).
$$
This gives the analogue of \eqref{eq:M3}:
\begin{equation}\label{eq:M3a}
 \dot M = - C_\xi M^{1-\frac{1}{\beta_0}}
 \left(c_0+c_2M^\kappa + x\Psi(x,M) \right)^{\frac{1}{\kappa \beta_0} + \frac{1}{\kappa \beta_2}}
\left( \frac{ 1+\psi(x,y) }{1+\psi(\xi,\eta) } \right)^{1-\frac{1}{\kappa \beta_0} - \frac{1}{\kappa \beta_2}},
\end{equation}
where $C_\xi$ is as in \eqref{eq:Exi}. 
To estimate $\tilde T$, we take some increasing function 
$\delta \leq \rho(\delta) \leq \delta^{1/2}$ such that $\delta = o(\rho(\delta))$ and divide the 
trajectory $\tilde \Phi^t(z_0) = (\tilde x(t), \tilde y(t))$  of $z_0$ into three parts:
\begin{equation}\label{eq:T1T2}
\tilde T_1 = \min\{ t > 0 : \tilde y(t) = \rho(\delta) \},
\qquad
\tilde T_2 = \max\{ t < \tilde T :  \tilde x(t) = \rho(\delta) \},
\end{equation}
and let $T_1, T_2$ be the analogous quantities for the unperturbed trajectory.
We compute
$$
T_1 = \int_{y_0}^{\rho(\delta)} \frac{dy}{\dot y} = 
\int_{y_0}^{\rho(\delta)} \frac{dy}{-y(a_2 x(y)^{\kappa} + b_2 y^{\kappa})} 
= O(\rho(\delta)^{-\kappa}).
$$
Similarly, using $\tilde x(y)/x(y) = 1+O(\psi(\delta))$ as in \eqref{eq:x(y)},
\begin{eqnarray*}
\tilde T_1 - T_1 &=& 
\int_{y_0}^{\rho(\delta)}\frac{1+O(y)}{-y(a_2 \tilde x(y)^{\kappa} + b_2 y^{\kappa})} 
-  \frac{1}{-y(a_2 x(y)^{\kappa} + b_2 y^{\kappa})} \ dy \\
&=& \int_{y_0}^{\rho(\delta)}\frac{O(y)( 1+O(\psi(\delta)))}{-y(a_2 x(y)^{\kappa} + b_2 y^{\kappa})} \ dy = O(\rho(\delta)^{1-\kappa}).
\end{eqnarray*}
This gives
\begin{equation}\label{eq:tildeT1}
\tilde T_1 = T_1(1+ O(T_1^{-\frac1{\kappa}})) \quad \text{ and } \quad
\tilde T - \tilde T_2 = (T-T_2) (1+O(\rho( (T-T_2)^{-\frac1{\kappa}}) ))
\end{equation}
by a similar computation for $T-T_2 = \int_{\rho(\delta)}^{\zeta_0}\frac{dx}{\dot x}$, etc.

Finally, for $\tilde T_1 < t < \tilde T_2$, we have $\psi(x, y) = O(\delta^{\alpha_*}, \delta \log(1/\delta))$ by \eqref{eq:psidelta2},
and $x\Psi(x,M) = O(x + y M^\kappa) = (1+M^\kappa) O(\rho(\delta))$.
Therefore 
\begin{eqnarray*}
\tilde T_2 - \tilde T_1 &=& \int_{M(\tilde T_2)}^{M(\tilde T_1)} 
\frac{ \left(1+\frac{ O(x+yM^\kappa) }{c_0+c_2 M^\kappa} \right)^{\frac{1}{\kappa \beta_0} + \frac{1}{\kappa \beta_2}} }
{ C_\xi M^{1-\frac{1}{\beta_0}} (c_0+c_2 M^\kappa)^{\frac{1}{\kappa \beta_0} + \frac{1}{\kappa \beta_2}} }
\left( \frac{ 1+\psi(x,y) }{1+\psi(\xi,\eta) } \right)^{\frac{1}{\kappa \beta_0} + \frac{1}{\kappa \beta_2}-1}\ dM\\
&=& \int_{M(\tilde T_2)}^{M(\tilde T_1)} 
\frac{ 1+ O(\rho(\delta)) }
{ C_\xi M^{1-\frac{1}{\beta_0}} (c_0+c_2 M^\kappa)^{\frac{1}{\kappa \beta_0} + \frac{1}{\kappa \beta_2}} }
(1+O(\delta^{\alpha_*}, \delta \log(1/\delta) ) \ dM\\
&=& (T_2-T_1)(1+O(\rho(\delta), \delta^{\alpha_*}, \delta \log(1/\delta))  ).
\end{eqnarray*}
Choosing $\rho(\delta) = \delta \log(1/\delta)$,
and using \eqref{eq:deltaT} gives
$$
\tilde T_2 - \tilde T_1 = 
(T_2-T_1)(1+O(T^{-\frac{\alpha_*}{\kappa}} , T^{-\frac{1}{\kappa}} \log T) ).
$$
Combining this with \eqref{eq:tildeT1} 
gives
$\tilde T = T(1+O(T^{-\beta_*}, T^{-\frac{1}{\kappa}}  \log T))$ for 
$\beta_* =  \frac{1}{\kappa} \min\{1, \frac{a_2}{b_2}, \frac{b_0}{a_0} \}$.
The estimate of Proposition~\ref{prop:regvar} now gives 
$\tilde \xi(\eta, \tilde T) = 
\xi_0(\eta) \tilde T^{-\beta_2}(1+O(\tilde T^{-\beta_*}, \tilde T^{-\frac{1}{\kappa}} \log \tilde T))$ 
as claimed.
\end{proof}

\subsection{Regular variation of $\mu(\varphi > n)$: proof of Theorem~\ref{thm:regvarmu} }
\label{sec:proof}

In the next lemma, we make the step from $C^\infty$ diffeomorphism of the previous
section to $C^{\kappa+2}$ diffeomorphisms. 
The need for this approximation argument is that we do not know a priori
if a $C^{\kappa+2}$ diffeomorphism of form \eqref{eq:diffeo}
is indeed the time-$1$ map of $C^{\kappa+1}$ vector field of the form \eqref{eq:vfq2}.
It may well be so in specific cases, see \cite{DRR}.

\begin{lemma}\label{lem:xi}
 Let $f$ be $C^{\kappa+2}$ almost Anosov diffeomorphism of local form \eqref{eq:diffeo},
 with $a_0, a_2, b_0, b_2 > 0$ and $\Delta \neq 0$.
 Recall that $\beta_* = \frac{1}{\kappa}\min\left\{1,\frac{a_2}{b_2}, \frac{b_0}{a_0}\right \}$. 
  Then
$$
\tilde \xi(\eta,\tilde T) = \xi_0(\eta) \tilde T^{-\beta_2} 
(1 + O(\tilde T^{-\beta_*}, \tilde T^{-\frac{1}{\kappa}}\log T)).
$$ 
as $\N \owns T \to \infty$, and $\xi_0(\eta)$ is as in Proposition~\ref{prop:regvar}.
\end{lemma}

\begin{proof}
 Let $f$ be $C^{\kappa+2}$ almost Anosov diffeomorphism with local form \eqref{eq:diffeo}.
 Then there is a sequence $f_j$ of almost Anosov diffeomorphisms
 which coincide with $f$ outside $U \supset \Phi^{-1}(Q) \cap \Phi^1(Q)$, and
 are $C^\infty$ with local form \eqref{eq:diffeo} and converge to $f$ in $C^{\kappa+2}$-topology inside $U$.
 By Lemma~\ref{lem:straight} we can assume $\partial Q$ consist of local stable and unstable leaves
 of $f_j$ for each $j$.
 By \cite[Theorem B and consequence 1(ii) on page 36)]{DRR},
 there are $C^\infty$ vector fields of local form \eqref{eq:vfhot}, such 
 that the $f_j$ are the time-$1$ maps of their flows.
 By Proposition~\ref{prop:perturb},
 $\tilde\xi_j(\eta, T) = \xi_{0,j}(\eta) T^{-\beta_2} (1 + O(T^{-\beta_*}))$,
 where (as one can verify from the proof) the $O(T^{-\beta_*})$-terms depend only on the first $\kappa+1$ 
 derivatives of the vector field.
 Since $f_j \to f$ in the $C^{\kappa+2}$-topology, these terms are uniform in $j$, 
 say they are $ \leq A \max\{ T^{-\beta_*},T^{-\frac{1}{\kappa}} \log T\}$ for some $A > 0$ 
 independent of $j$ and $T$.
 Also $\xi_{0,j} \to \xi_0$ uniformly in $\eta$.
 
 Take $j_n$ so large that $|\xi_{0,j} - \xi_0|_\infty < T^{-1}$ for all $j \geq j_n$.
 Then the triangle inequality gives
 $$
 |\tilde\xi(\eta,T) - \xi_0(\eta) T^{-\beta_2}| \leq (A+1) \max\{ T^{-\beta_*},T^{-\frac{1}{\kappa}} \log T\}
 $$
 as required.
 \end{proof}

\begin{proofof}{Theorem~\ref{thm:regvarmu}}
Recall the definition of $\eta_0$ and $\eta_1$ from the proof 
of Theorem~\ref{thm:regvarmu_special}.
Lemma~\ref{lem:xi} yields the estimates for $\xi(\eta,n)$ for $\eta_0 \leq \eta \leq \eta_1$.
Every local unstable leaf $W^u$ intersects the stable leaf $W^s$ of $p$
in a unique point $(0,y)$, so we parametrise these local unstable leaves as
$W^u(y)$. Also the conditional measure $\mu^u_{W^u(y)}$ is absolutely continuous
w.r.t.\ Lebesgue, so we can write $d\mu^u_{W^u(y)} = h(x,y) dm^u_{W^u(y)}$,
and in fact $h(x,y)$ is differentiable in $x$.
We can decompose $d\mu = d\mu^u_{W^u(y)} d\mu^s$.
Then we get
 \begin{eqnarray*}
  \mu(\varphi > n) &=& \int_{\eta_0}^{\eta_1}  \int_{W^u(y)} 1_{\{\varphi > n\}} d\mu^u_{W^u(y)} \ d\mu^s \\
  &=& \int_{\eta_0}^{\eta_1}  \int_{W^u(y)} h_{W^u(y)}(x,y) \ dm^u_{W^u(y)}(x) \ d\mu^s \\
  &=& \int_{\eta_0}^{\eta_1}  \int_0^{\xi_0(y)n^{-\beta_2}(1+O(n^{-\beta_*}, n^{-\frac1\kappa}\log n))} 
  h_{W^u(y)}(0,y) + \\
&& \qquad  \frac{\partial}{\partial x} h_{W^u(y)}(0,y) x +O(x^2) \ dm(x) \ d\mu^s \\
  &=& n^{-\beta_2} \left(1+O(\max\{ n^{-\beta_*}, n^{-\beta_2}, n^{-\frac1\kappa}\log n \}) \right) 
  \int_{\eta_0}^{\eta_1} \xi_0(y) h_{W^u(y)}(0,y)\, d\mu^s  .
 \end{eqnarray*}
This proves the result for 
$C_0 = \int_{\eta_0}^{\eta_1} \xi_0(y) h_{W^u(y)}(0,y) \ d\mu^s$.
\end{proofof}

\section{Banach spaces estimates}\label{sec-Bspace}

In this section we verify the hypotheses in Section~\ref{sec:set-up} for the maps described in Section~\ref{sec-LT}.

\paragraph{Convention on the use of constants:} Unless otherwise specified, throughout this section
$C$ will denote a positive constant that might vary from line to line.

\vspace{-1ex}

\subsection{Notation and definitions}\label{subsec-admleaves}

Since $f:{\bT}^2 \to {\bT}^2$ is Markov, $F=f^\varphi:Y \to Y$ is also Markov. Indeed, let $\cP$ be the finite  Markov partition for $f$ 
(into rectangles, including $P_0$).
Then $\cY=\cup_{n\geq 1}\{\varphi=n\} \cap \cP^n$, where $\cP^n$ is the $n$-th refinement of $\cP$, 
is a Markov partition for $F$.
We let $Y_j$ be the elements of $\cY$ indexed such that there is $j_0$ such that
$\{ \varphi = n \} = Y_{j_0+n}$.
Note that these sets are small "rectangles" with 'wavy' boundaries,
namely two pieces of stable and two pieces of unstable curve. 
The stable lengths of the elements $Y_j$ are bounded away from zero, and so will be the admissible leaves
below.
Also, the image partition $F(\cY):=\{Y_j'\}$ consists of "rectangles" with wavy boundaries,
again two pieces of stable and two pieces of unstable curve.

For $n\geq 0$, let  $\cY_n=\{Y_{n,j}\}$ be the Markov partition associated with $(Y, F^n)$.  
Since $F$ is invertible, we have $F^{-n}(\{Y_{n,j}'\})=\{Y_{n,j}\}$. 
The map $F^n$ is smooth in the interior of each element of $\cY_n$. 

Singularities for the map $F$ are solely created via inducing and are placed on the 'wavy' boundaries. 
We let $\cS^{\pm n}$ be the set of such singularities for $F^{\pm n}$.

\paragraph{Admissible leaves, distance between leaves:}

Throughout, $\Sigma$ denotes the  set of {\em admissible leaves}, which consists of 
maximal stable leaves $W_{Y_j}^s$ in  partition elements $Y_j$ of $\cY$.
Such leaves can be conveniently described via charts $\chi_j:[0,L_u(Y_j)] \times [0,1] \to Y_j$, where $Y_j$ 
is an element of the Markov partition $\cY$.
Let $L_u(Y_j)$ be the length of the (largest) unstable leaf in $Y_j$ and  assume that $\chi_j^{-1}$ maps 
the 'wavy' boundaries to the boundary of the rectangle $[0,L_u(Y_j)]\times [0,1]$. 
Therefore we can stipulate that the distortion of $\chi_j$ is bounded, uniformly in $j$.
The other stable and unstable leaves in $Y_j$ map
under $\chi_j^{-1}$ to roughly "vertical and horizontal"\footnote{These are still somewhat 'wavy' curves 
in the square $[0,L_u(Y_j)]\times [0,1]$.}.
More precisely, for any leaf $W_{Y_j}^s$ in a partition element $Y_j$, 
\begin{equation}\label{eq-chart}
\chi_j^{-1}(W_{Y_j}^s)=\{(g_{Y_j,W_{Y_j}^s}(\eta),\eta): \eta\in [0,1]\},
\end{equation}
where $g_{Y_j,W^s_{Y_j}}:[0,1]\to[0,L_u(Y_j)]$ is $C^4$ (since by assumption $f$
is four times differentiable). 
Writing $W^u$ for maximal unstable leaves, also let
\begin{equation}\label{eq:L}
L := \inf_j \inf_{W^u \subset Y_j} |F(W^u)|,
\end{equation}
which is positive because the images of the elements $\{ \phi = n\}$ have uniformly long unstable lengths.

When there is no risk of confusion, we write $W:=W_{Y_j}^s$. 
Also, we note that in the above notation,
for any $(x,y)\in W_{Y_j}^s$, there exists $\eta\in [0,1]$ such that
$\chi_j^{-1}(W_{Y_j}^s)(x,y)=(g_{Y_j,W_{Y_j}^s}(\eta),\eta)$. 
When there is no risk of confusion, we write $g:=g_{Y_j,W_{Y_j}^s}$.

This definition of $\Sigma$ differs from the one in \cite{DemersLiverani08} (being closer to the simplification in~\cite{LT}) and allows for simpler arguments
similar to the ones in~\cite{LT}. This is possible  due to the Markov structure of $F$.

Given the representation~\eqref{eq-chart}, we define the distance between leaves $W,\tilde W\in\Sigma$ 
such that $W \in Y_k$ and $\tilde W \in Y_{\tilde k}$, by
$$
d(W,\tilde W)= 
 \begin{cases}
 \sup_{\eta\in [0,1]}|g(\eta)-\tilde g(\eta)| & \text{ if } k = \tilde k;\\
 \sup_{\eta\in [0,1]}|g(\eta)-L_u(Y_k)-\tilde g(\eta)| + \sum_{j=k+1}^{\tilde k-1} L_u(Y_j) & \text{ if } j_0+2 \leq k < \tilde k;\\
 \infty & \text{ otherwise.}
 \end{cases}
$$
Recall here that $j_0$ is such that $\{ \varphi = n \} = Y_{j_0+n}$ and an empty sum $\sum_{j=k+1}^k$ is $0$ by convention.
\vspace{-1ex}
\paragraph{Uniform contraction/expansion, distortion properties:}
Since $f$ satisfies Definition~\ref{def-AlmAn} and Remark~\ref{rem:add_def}, $F$ is hyperbolic. That is, there exist two 
transversal families of stable and unstable cones $y\to C^s(y), C^u(y)$ such that items i) and ii) hold with $F$ instead of $f$,
for all $y\in Y\setminus\cS^{+1}$.
There exist $\lambda>1$ and $C>0$ such that for all $n\geq 0$, $j \geq 1$
for all $y\in Y_{n,j}$.

\begin{equation}\label{eq-dist}
\|DF^n(y)v\|\geq C \lambda^n\|v\|,\forall v\in C^u(y)\ \mbox{ and }\ 
\|DF^{-n}(y)v\|\geq C \lambda^n\|v\|,\forall v\in C^s(y),
\end{equation}
where $\|\,\|$ is the Euclidean norm on the tangent space $T_y(Y)$.

Let $J_W F^n$ be the Jacobian of $F^n$ along the stable leaf $W$ and let $J_uF^n$ 
be the Jacobian of $F^n$ in the unstable direction.
Note that for any $y\in Y\setminus S^{+n}$, 
\begin{equation}\label{eq:Ctheta}
|DF^n(y)| := |\det(DF^n(y))|=C_\theta(y)J_W F^n(y)J_uF^n(y),
\end{equation}
where $C_\theta(y)$ is a number depending on the angle $\theta$ between the stable 
leaf $W$ and unstable leaf
 at the point $y$.

Since the family of admissible leaves $\Sigma$ is transversal to the unstable leaves
(with a uniform lower bound on their angle), 
there exist $\lambda\in (0,1)$ and $C>0$ independent of $y$ and $W$ 
such that for all $n\geq 0$ such that $|DF^n|$ is defined,
\begin{equation}\label{eq-dist1}
\Big| |DF^n|^{-1} J_W F^n \Big|_\infty\leq C\lambda^{-n}, \quad |J_W F^n |_\infty\leq C\lambda^{-n}
\end{equation}
Because $F$ is hyperbolic, uniformly on all $Y_j$, we have by \eqref{eq-dist} and
e.g.\ \cite[Appendix A]{DemersLiverani08} that there exists some $C, C', C''>0$
such that for all $Y_j$ and all $z,w$ in the same connected component of $Y_{n,j} \setminus\cS^{+n} \subset Y_j$ and 
all $W\in\Sigma$, 
\begin{equation}\label{eq-distort1}
\begin{cases}
\Big| \frac{|DF^n(z)|}{|DF^n(w)|}-1\Big| \leq C\max\{d(z,w), d(F^n(z),F^n(w))\},\\[3mm]
\Big| \frac{J_WF^n(z)}{J_WF^n(w)}-1\Big|\leq C'\max\{d(z,w), d(F^n(z),F^n(w))\},\\[3mm]
\Big| \frac{J_uF^n(z)}{J_uF^n(w)}-1\Big|\leq C''\max\{d(z,w), d(F^n(z),F^n(w))\},
\end{cases}
\end{equation}
where
\begin{equation}\label{eq-distance}
d(z,w)=\|\chi_j^{-1}(z)-\chi_j^{-1}(w)\|
\end{equation}
with $\|\,\|$ denoting the Euclidean distance on $[0, L_u(Y_j)] \times [0,1]$. 
Recall that $\cY_n=\{Y_{n,j}\}$ is the Markov partition for $F^n$.

\begin{lemma}
There is a constant $C > 0$ such that for every $W \in \Sigma$ and 
$W_j = F^{-n}(W) \cap Y_{n,j}$,
\begin{equation}\label{eq-dist3}
\Big| |DF^n|^{-1}J_{W_j} F^n\Big|_{\infty}\leq C m(Y_{n,j})\ 
\mbox{ and }\ \sum_{W_j\in{\cY_n}}\Big| |DF^n|^{-1}J_{W_{n,j}} F^n\Big|_{\infty}<\infty.
\end{equation}
\end{lemma}

\begin{proof}
The Markov partition $\mathcal Y$ of the induced map $F: Y \to Y$
has elements of the form $\{ \varphi = n\}$, $n \geq 2$, so still
the stable lengths of partition elements are $\geq L$ (see \eqref{eq:L}), and this remains true
for elements in ${\mathcal Y}_n$. Analogously, the unstable lengths of the 
elements of the image partition ${\mathcal Y}'_n$ are also $\geq L$.

Therefore $m(Y_{n,j}) \approx L_u(Y_{n,j})$ and $L_u(F^n(Y_{n,j})) \approx 1$. 
Let $W \in \Sigma$ be a stable leaf and $W_j := F^{-n}(W) \cap Y_{n,j}$.
Recall that $|\det(DF^n(y))|=C_\theta(y)J_W F^n(y)J_uF^n(y)$,
where $C_\theta(y)$ is as in \eqref{eq:Ctheta}.
These together with the distortion control \eqref{eq-distort1} of the unstable derivatives, give
$$
\Big| |DF^n|^{-1} J_{W_j}F^n \Big|_\infty \approx \Big| \Big( J_uF^n|_{W_j} \Big)^{-1} \Big|_\infty \approx
\frac{L_u(Y_{n,j})}{L} \approx L_u(Y_{n,j}) \approx m(Y_{n,j}).
$$
Now to sum over all pieces $W_j$ of $F^{-n}(W)$, note that each $W_j$ lies in a separate
element $Y_{n,j}$. Therefore, there is $C > 0$ such that
$$
\sum_j \Big| |DF^n|^{-1} J_{W_j}F^n \Big|_\infty \leq C \sum_j m(Y_{n,j}) \leq m(Y) < \infty,
$$
as required.
\end{proof}

\vspace{-1ex}
\paragraph{Test functions:}
In what follows, for $W\in\Sigma$ and $q\leq 1$ we denote by $C^q(W)$ the Banach space of complex
valued functions on $W$ with  H\"older exponent $q$ and norm
\[
|\phi|_{C^q(W)}=\sup_{z\in W} |\phi(z)|+\sup_{z,w\in W}\frac{|\phi(z)-\phi(w)|}{d(z,w)^q}.
\]
Note that for a given $W_j\in\Sigma \cap Y_j$, $C^q(W)$ is isomorphic to $C^q([0,1])$ 
via the identification of the domain given by the representation via charts in~\eqref{eq-chart}. 
Throughout we will use such an identification without further notice. In particular, given $\phi\in C^q([0,1])$ 
we still call $\phi$ the corresponding function in $C^q(W_{Y_j}^s)$ and
using~\eqref{eq-chart} we write
\begin{equation}
\label{eq-integrleave}
\int_{W_j} h\phi\, dm=\int_{0}^{1} h\circ \chi_j(g(\eta) ,\eta) \phi_g(\eta)\sqrt{1+g'(\eta)^2}\, d\eta,
\end{equation}
where $\phi_g(\eta)=\phi(g(\eta),\eta)$ and $\sqrt{1+g'(\eta)^2}\, d\eta$ is 
essentially the arc length (Lebesgue) measure on $\chi^{-1}_j(W_j)$.
\begin{remark}
Note that we use $m$ both for the one dimensional and two dimensional Lebesgue measure. 
\end{remark}
\vspace{-1ex}

\paragraph{Definition of the norms:}

Given $h\in C^1(Y,\C)$, define the \emph{weak norm} by
\begin{equation}\label{eq-weaknorm}
\|h\|_{\cB_w}:=\sup_{W\in\Sigma}\;\sup_{|\phi|_{ C^1(W)}\leq 1 }\int_W h\phi\, dm.
\end{equation}

Given $q\in [0,1)$ we define the \emph{strong stable norm} by

\begin{equation}\label{eq-strongnormst}
\|h\|_s:=\sup_{W\in\Sigma}\;\sup_{|\phi|_{ C^q(W)}\leq 1 }\int_W h\phi\, dm.
\end{equation}

Finally, recalling~\eqref{eq-dist} we define the \emph{strong unstable norm} by
\begin{equation}\label{eq-strongnormunst}
\|h\|_u:=\sup_\ell \sup_{W,\tilde W\in\Sigma \cap Y_\ell}
\sup_{\stackrel{|\phi|_{C^1(W)}, |\phi|_{C^1(\tilde W)} \leq 1}{d(\phi,\tilde\phi) \leq d(W,\tilde W)}}
\frac{1}{d(W,\tilde W)}\left|\int_{W} h\phi\, dm -\int_{\tilde W} h\phi\, dm\right|,
\end{equation} 
where 
$$
d(\phi,\tilde\phi) = | \phi \circ \chi_\ell( g(\eta), \eta) - \tilde\phi \circ \chi_\ell(\tilde g(\eta), \eta)|_{C^1([0,1])}.
$$
The \emph{strong norm} is defined by $\|h\|_{\cB}=\|h\|_s+\|h\|_u$.

\paragraph{Definition of the Banach spaces:}
We will see in Lemma~\ref{lemma-embed} that $\|h\|_{\cB_w}+\|h\|_{\cB}\leq C\|h\|_{C^1}$. We then define $\cB$ to be the completion of $C^1$ in the strong norm and $\cB_w$ to be the completion in the weak norm.

The spaces $\cB$ and $\cB_w$ defined above are simplified versions of functional spaces defined
in~\cite{DemersLiverani08} (adapted to the setting of~\eqref{eq:diffeo}).
The main difference in the present setting is the simpler definition of admissible leaves and the absence of a control on short leaves. This is possible due to the Markov structure of the diffeomorphism.

\subsection{Embedding properties: verifying (H1)(i)}

The next result shows that (H1)(i) holds for  $\cB, \cB_w$ as described above.

\begin{lemma}{\cite[Lemma 7.2]{LT}}
\label{lemma-embed} For all $q\in (0,1)$ in definition \eqref{eq-strongnormst} we have\footnote{Here the inclusion is meant to signify a continuous embedding of Banach spaces.}
\[
C^1\subset \cB\subset \cB_w\subset(C^1)'.
\]
Moreover, the unit ball of $\cB$ is relatively compact in $\cB_w$.
\end{lemma}
\begin{proof}
By the definition of the norms it follows that $\|\cdot\|_{\cB_w}\leq \|\cdot\|_s\leq\|\cdot\|_{\cB}$.
From this the inclusion $\cB\subset \cB_w$ follows. 

Using~\eqref{eq-integrleave}, for each $h, \phi\in C^1$ and for each $j$ and $W, \tilde W\in Y_j\in\cY$, 
\begin{align}\label{eq-emb1}
\nonumber &|\int_{W} h\phi\, dm-\int_{\tilde W} h\phi\, dm| \\
\nonumber &=
\Big|\int_{0}^{1} h\circ \chi_j(g(\eta),\eta))\sqrt{1+g'(\eta)^2}\, \phi_g(\eta)
-h\circ \chi_j(\tilde g(\eta),\eta))\sqrt{1+\tilde g'(\eta)^2}\, \phi_{\tilde g}(\eta)\, d\eta\Big|\\
\nonumber &\leq 
\Big|\int_{0}^{1} \Big(h\circ \chi_j(g(\eta),\eta))-h\circ \chi_j(\tilde g(\eta),\eta)\Big)\sqrt{1+g'(\eta)^2}\, \phi_g(\eta)\, d\eta\Big|\\
\nonumber & \quad +\Big|\int_{0}^{1} h\circ \chi_j(\tilde g(\eta),\eta)\Big(\sqrt{1+g'(\eta)^2}-\sqrt{1+\tilde g'(\eta)^2}\Big) \phi_g(\eta)\, d\eta\Big|\\
\nonumber & \quad +\Big|\int_{0}^{1} h\circ \chi_j(\tilde g(\eta),\eta) \sqrt{1+\tilde g'(\eta)^2}(\phi_g(\eta)-\phi_{\tilde g}(\eta))\, d\eta\Big|\\
\nonumber &\leq C\|h\|_{C^1}\|\phi\|_{C^0}(\sup_{\eta\in [0,1]}|\chi(g(\eta),\eta)-\chi(\tilde g(\eta),\eta)|+|g'-\tilde g'|_\infty)\\
&\nonumber \quad + C' \| h \|_{C^0} \sup_{\eta \in [0,1]}|\phi_g(\eta) - \phi_{\tilde g}(\eta)|\\
&
\leq C( \|h\|_{C^1}\|\phi\|_{C^0} + \| h \|_{C^0} \| \phi \|_{C^1} )d(W, \tilde W).
\end{align}
The above implies that $\|h\|_u\leq C \|h\|_{C^1}$. Thus $C^1\subset \cB$.

The other inclusion is an immediate consequence of Proposition~\ref{prop-embed}, an analogue of~\cite[Lemma 3.3]{DemersLiverani08}.
The injectivity follows from the injectivity of the standard inclusion of $C^1$ in $(C^1)'$.

Next, we need to show that the unit ball $B_1$ of $\cB$ has compact closure in $\cB_w$. 
Note that it suffices to show that it is 
totally bounded, i.e., for each $\ve>0$, it can be covered by finitely many $\ve$-balls in the $\cB_w$ norm. 

For any $\ve>0$, let $\cN_\ve$ be a finite collection of leaves in $\Sigma$ 
such that for any $W\in\Sigma$, there exists $\tilde W\in \cN_\ve$ with $d(W,\tilde W)\le\ve$. 
Let $N_\ve := \#(\cN_\ve) \leq K/\ve$ for some constant $K$ independent of $\ve$.
Say $\cN_\ve = \{ \tilde W_j \}_{j=1}^{N_\ve}$.
By~\eqref{eq-emb1} together with $\|h\|_u\leq C |h|_{C^1}$, for every $W\in\Sigma$ and test 
function $\phi\in C^1$ with  $|\phi|_{C^1}\leq 1$,
there exists $\tilde W\in \cN_\ve$ such that
\[
\left|\int_Wh\phi\, dm-\int_{\tilde W}h\phi\, dm \right|\leq \ve\|h\|_u.
\]
On the other hand, by the Arzel\`a-Ascoli theorem, for each $\ve > 0$ there exist a
finite collection $\{\phi_i\}_{i=1}^{M_\ve}\subset C^1([0,1])$ which is $C^q$ $\ve$-dense in the unit ball of $C^1$.
Accordingly, for each $\phi\in C^1$ with $|\phi|_{C^1}\leq 1$ and for every $W\in\Sigma$ and $\tilde W\in \cN_\ve$ 
such that $d(W,\tilde W)\le\ve$, there exists $\phi_i$ such that
\begin{align}\label{eq:compact-short}
\nonumber\left|\int_Wh\phi\, dm-\int_{\tilde W}h\phi_i\, dm \right|
&\leq\left|\int_W h(\phi-\phi_i)\, dm\right|+ \left|\Big(\int_W-\int_{\tilde W}\Big)h\phi_i\, dm\right|\\
&\leq C \|h\|_{\cB_w}\, |\phi-\phi_i|_{C^1}
+C \ve\|h\|_u\leq 2C \ve\|h\|_{\cB}.
\end{align}
Let $K_\ve:\cB_w\to\C^{M_\ve N_\ve}$ 
be defined by $[K_\ve(h)]_i=\int_{\tilde W_j}h\phi_i\,dm$ for some $j \in \{ 1, \dots, N_\ve\}$. 
Clearly, $K_\ve$ is a continuous map.
Since the image of the unit ball $B_1$ under $K_\ve$ is contained in 
$\{a \in \C^{M_\ve N_\ve}\;:\; |a_{ij}|\leq 1\}$, it has a compact closure. 
Hence, there exists finitely many $a^k \in \C^{M_\ve N_\ve}$ such that the sets
\[
U_{k,\ve}=\left\{h\in\cB_w\;:\; \left|\int_{\tilde W_j}h\phi_i\, dm-a_{ij}^k\right|\leq \ve, \,\forall i,j\right\}
\]
cover $B_1$. To conclude note that if $h_1,h_2\in U_{k,\ve}$, then, by \eqref{eq:compact-short}, 
there exist $i$ and $\tilde W$ such that
\[
\begin{split}
\left|\int_W(h_1-h_2)\phi\, dm \right|
&\leq \left|\int_{\tilde W_j}(h_1-h_2)\phi_i\, dm \right|+2C\ve\|h_1-h_2\|_{\cB}\\
&\leq \left|\int_{\tilde W_j}h_1\phi_i\, dm-a_{ij}^k\right|+
\left|\int_{\tilde W_j}h_2\phi_{i}\, dm-a_{ij}^k\right|+4C\ve\\
&\leq 4(C+1)\ve.
\end{split}
\]
This means that each $U_{k,\ve}$ is contained in a $4(C+1)\ve$-ball in the $\cB_w$ norm and the 
conclusion follows.
\end{proof}

\subsection{Verifying (H2)}
The first step in verifying (H2) is

\begin{prop}\label{prop-embed}
For any $j\geq 1$, let  $Y_j\in\cY$. Let $W,\tilde W\in \Sigma\cap Y_j$ and 
write $W=\{\chi_j(g(\eta),\eta) : \eta\in [0,1]\}$ 
and  $\tilde W=\{ \chi_j(\tilde g(\eta),\eta) : \eta\in [0,1]\}$.
Let $E=\{\chi_j(y,\eta): g(\eta)\leq y \leq \tilde g(\eta), \eta\in [0,1]\}$ 
and note that $\Id_E|_{W^*}$ is constant on any $W^*\in \Sigma\cap Y_j$.
Then for any such set $E$,  for all $\phi\in C^1(Y_j)$
and for all $h\in\cB_w$, we have $\Id_E h\in \cB_w$. Moreover, 
\[
|\langle\Id_E h, \phi\rangle| \leq  \|h\|_{\cB_w} |\phi |_{C^1} m(E).
\]
\end{prop}

\begin{proof}
First note that on any $Y_j\in\cY$, the Lebesgue measure $m$ can be decomposed according to 
the collection $\cW_j=\{W_\ell\}$ of stable leaves on $Y_j$.
That is, there exists a measure $\nu$ on $\cW_j$ such that for any $\psi\in L^1(m)$
\begin{align*}
\int_{Y_j} h\, dm=\int_{\cW_j}d\nu\int_{W_\ell} h\, dm.
\end{align*}
Let  $h\in C^1$ and consider $v\in L^\infty$ such that $v|_W$ is constant for any $W\in\Sigma$. 
Using~\eqref{eq-integrleave}, we have that
\begin{align*}
\left|\int_Y  h v\phi\, dm \right|&\leq\sum_j \int_{\cW_j} |v| d\nu\left|\int_{W_\ell} h\phi\,dm \right|\\
&\leq \sum_j \int_{\cW_j} |v| d\nu_j \left|\int_0^1  h\circ \chi_j(g(\eta),\eta))\tilde\phi(\eta) 
\sqrt{ 1 + g'^2(\eta)} \, d\eta\right|\\
&\leq \|h\|_{\cB_w} |\phi |_{C^1} \|v\|_{L^1}.
\end{align*}
To see that $\Id_E h\in\cB_w$ note that $E$ is a union of unstable leaves (between $W$ and $\tilde W$) and that the following are the only possibilities:

\begin{itemize}
\item If $W\in\Sigma$ but $W\notin E$ then  $\int_W \Id_E h\phi \,dm=0$;
\item If $W\in\Sigma\cup E$ then $\int_W 1_E h\phi \,dm=\int_W h\phi \,dm$, so  $\|\Id_E h\|_{\cB_w}\leq  \|h\|_{\cB_w}$.
\end{itemize}
Also, from the previous displayed equation with $v=\Id_E$, we obtain that
\[
|\langle\Id_E h, \phi\rangle|\leq \|h\|_{\cB_w}\|\phi\|_{C^1(Y_j)} m(E)
\]
as required.~\end{proof}

Note that the connected components of $\{\varphi = n\}$ satisfy the assumption on the set $E$ in 
the statement of Proposition~\ref{prop-embed}. Therefore
\[
\left|\int_E h\, dm\right|\leq  \|h\|_{\cB_w} m(E).
\]
We still need to argue that the same is true with $\mu$ instead of $m$ on the right hand side. 
Let $\cW^u$, $\cW^s$ be the collection of unstable, stable, respectively, leaves restricted to $E$.
There exists  measures $\nu^s$ on $\cW^s$ and $\nu^u$ on $\cW^u$ such that for any $W^s\in \cW^s$ and 
$W^u\in \cW^u$,
\begin{align*}
\Big|\int_E h \, dm\Big|&=\Big|\int_{\cW^s}\, d\nu^s \int_{W^s} h\, dm|_{W^s}\Big|\leq \|h\|_{\cB}\Big|
\int_{\cW^s}\, \int_{W^s} \, dm|_{W^s}\, d\nu ^s\Big|\\
&= \|h\|_{\cB}\Big|\int_E \, dm\Big|=  \|h\|_{\cB}\Big|\int_{\cW^u}\, \int_{W^u} \, dm|_{W^u}\, d\nu^u\Big|
\end{align*}
Since $\mu$ is absolutely continuous w.r.t.\ $m$ on the unstable leaves with density bounded away from $0$, 
there exist $C$ such that $C^{-1}\leq \frac{d\mu|_{W^u}}{dm|_{W^u}}$. Thus,
\begin{align*}
\Big|\int_E h \, dm \Big|\leq \|h\|_{\cB}\Big|\int_{\cW^u}\, \int_{W^u} \, d\mu|_{W^u}\, d\nu^u\Big|
=C\|h\|_{\cB}\Big|\int_E \, d\mu\Big|=C\|h\|_{\cB}\mu(E).
\end{align*}

\subsection{Transfer operator: definition}\label{subsec-Trop}

If $h\in L^1(m)$, then $R:L^1(m)\to L^1(m)$ acts on $h$ by
\[
\int_Y Rh\cdot v\, dm=\int_Y h\cdot v\circ F\, dm,\quad v\in L^\infty.
\]
By a change of variables we have
\begin{equation}\label{eq:transfer-L1}
Rh=\Id_Y h\circ F^{-1}|DF^{-1}|,
\end{equation}
where $|DF^{-1}|=|\det(DF^{-1})|$.
Note that, in general, $R C^1\not\subset C^1$, so it is not obvious that the operator $R$ has any chance of being well defined in $\cB$.
The next lemma addresses this problem, using the existence of the Markov partition.

\begin{lemma}\label{lem:R-well-def}
With the above definition, $R(C^1)\subset \cB$.
\end{lemma}

\begin{proof}
Using the notation introduced at the beginning of Section~\ref{subsec-admleaves}, 
we have $F(Y)=\cup_j Y_j'$. Moreover, both $F^{-1}$ and $\det(DF^{-1})$ are $C^1$ on each $\overline{Y_j'}$.
For each $j\in\bN$, $Y_j'$ is bounded by the stable curves $\gamma_0$, $\gamma_1$ defined via~\eqref{eq-chart} by
\[
\gamma_0(\eta)=\chi_j((g_j^0(\eta),\eta)), \quad \gamma_1(\eta)=\chi_j((g_j^1(\eta),\eta)), \eta\in[0,1].
\]
In particular, the above representation implies  that $|\gamma_0'|_\infty, |\gamma_1'|_\infty<\infty$. Thus,
we can consider a sequence of $\bar\psi_{n}\in C_0^1(\bR, [0,1])$ that converges monotonically to 
$\Id_{[0,1]}$.

Next, note that for every stable leaf $W \in \Sigma$,  $F^{-1}W=\cup_j W_j$ 
where $W_j=F^{-1}W\cap Y_j$ are (possibly infinitely many) stable leaves. With this notation, given the sequence $\bar\psi_n$
introduced above,  we define $\tilde \psi_{n,W_j}(\chi_j(g(\eta)),\eta) := \bar\psi_n(\eta)$.
With these specified, we further define the function 
$\psi_n=\tilde \psi_{n,W_j}(g(\eta),\eta)\circ F^{-1}\cdot \Id_{F(Y)}$ and note that 
$\psi_n$ is smooth and converges monotonically to $\Id_W$. For $h\in C^1$, 
let
\[
\cH_n = \psi_n h \circ F^{-1} |DF^{-1}| \in C^1,
\]
and compute that\footnote{ Since $W_j\subset Y$,  $F(W_j)\subset F(Y)$. 
Thus, when restricted on $W_j$, $\Id_{F(Y)}\circ F|_{W_j}=1_{W_j}$.}
\[
\begin{split}
\left|\int_W[Rh-\cH_n]\phi\right|\, dm &\leq \sum_j\int_{W_j} |h|\, \Big||DF|^{-1}J_{W_j}F\Big|\, |\phi\circ F|\, 
|1_{W_j}-\tilde \psi_{n,W_j}|\, dm \\
&\leq |h|_\infty\|\phi\|_\infty\sum_j  \Big||DF|^{-1}J_{W_j}F\Big|_\infty\int_{W_j} 
|1_{W_j}-\tilde\psi_{n, W_j}|\, dm.
\end{split}
\]
By~\eqref{eq-integrleave} and the fact that $\tilde \psi_{n,W_j}(g(\eta),\eta)=\bar\psi_{n}(\eta)$,
\[
\int_{W_j} |1_{W_j}-\tilde \psi_{n,W_j}|\leq \int_{0}^{1} |1-\bar\psi_n(\eta)|\sqrt{1+g'(\eta)^2}\, d\eta\leq C\int_{0}^{1} |1-\bar\psi_n(\eta)|\, d\eta.
\]
Thus,
\[
\begin{split}
\left|\int_W[Rh-\cH_{n}]\phi\, dm\right|
\leq C\|h\|_\infty\|\phi\|_\infty\int_{0}^{1} |1-\bar\psi_n(\eta)|\, d\eta\sum_j \, 
\Big||DF|^{-1}J_{W_j}F\Big|_\infty.
\end{split}
\]
By~\eqref{eq-dist3}, the sum is convergent and thus, $|\int_W[Rh-\cH_n]\phi\, dm|$  converges to zero as $n\to\infty$. 
As a consequence, $\cH_n$ converges to $Rh$ in $\cB_w$ and $\lim_{n\to\infty}\|Rh-\cH_n\|_s=0$. 

It remains to check the unstable norm. Let $\phi$ such that $|\phi|_{C^1(Y)}\leq 1$. Using~\eqref{eq-integrleave},
for any $W,\tilde W\in Y_j$ and for any $j\geq 1$ we compute that
\[
\begin{split}
&\Big|\int_{W}[Rh-\cH_n]\phi\, dm -\int_{\tilde W}[Rh -\cH_n]\phi\, dm \Big|\\
&\leq\sum_{j}\int_0^1\Big|h\circ\chi_j(g_j(\eta),\eta)|DF\circ\chi_j(g_j(\eta),\eta)|^{-1}J_{W_j}F\circ\chi_j(g_j(\eta),\eta)
\phi\circ F \circ\chi_j(g_j(\eta),\eta)\\
&\quad - h\circ\chi_j(\tilde g_j(\eta),\eta)|DF\circ\chi_j(\tilde g_j(\eta),\eta)|^{-1}J_{W_j}
F\circ\chi_j(\tilde g_j(\eta),\eta)\phi\circ F(\tilde g_j(\eta),\eta)\Big|\,\left|1-\bar\psi_n(\eta)\right|\,d\eta.
\end{split}
\]
Using that $h,\phi\in C^1$, recalling~\eqref{eq-dist1}\footnote{Here we also use that for $\eta\in [0,1]$,
$|DF(g_j(\eta),\eta)|^{-1}J_{W_j}F(g_j(\eta),\eta)|\leq \Big||DF|^{-1}J_{W_j}F\Big|_\infty$.}
 and the fact that $|g-\tilde g|_\infty=d(W,\tilde W)$ we obtain
that for some $C>0$,
\[
\left|\int_{W}[Rh-\cH_n]\phi\, dm -\int_{\tilde W}[Rh-\cH_n]\phi\, dm \right|
\leq Cd(W,\tilde W) \int_0^1\left|1-\bar\psi_n(\eta)\right|\, d\eta\, \sum_j  \Big||DF|^{-1}J_{W_j}F\Big|_\infty
\]
and we conclude using~\eqref{eq-dist3}.
\end{proof}

\subsection{Lasota-Yorke inequality and compactness: verifying (H5)(i) and (H1)(ii).}

The next lemma is the basic result on which all the theory rests.

\begin{prop}[Lasota--Yorke inequality] \label{prop-LSinequality} 
For each $z\in\overline\bD$, $n\in\bN$ and $h\in C^1(Y)$ we have  
\[
\begin{split}
&\|R(z)^n h\|_{\cB_w}\leq C|z|^n\|h\|_{\cB_w},\\[1mm]
&\|R(z)^n h\|_{\cB} \leq  \lambda^{-nq}|z|^n \|h\|_{\cB}+C|z|^n \|h\|_{\cB_w}.
\end{split}
\]
\end{prop}

\begin{proof}
Write $\vf_n=\sum_{k=0}^{n-1}\vf\circ F^k$ and note that $R(z)^n h=R^n(z^{\vf_n} h)$. 
Given $W\in\Sigma$, write $F^{-n}(W)=\cup_{j\in\bN} W_j$. Let $\cW=\{W_j\}_{j\in\bN}\subset \Sigma$ 
be the collections of the connected components of $F^{-n}(W)$;
each $Y_{n,j}\in\cY_n$ contains at most one $W_j$. 
For a test function $\phi\in C^1$, we have
\[
\begin{split}
\int_W (R(z)^n h)\phi\, dm&=\int_Wz^{\vf_n \circ F^{-n} }\Id_{F^n(Y)}h\circ F^{-n}|DF^{-n}|\phi\, dm \\
&=\sum_{W_j\in \cW}\int_{W_j} z^{\vf_n }\Id_{F^n(Y)}h\frac{J_{W_j}F^n}{|DF^n|}\phi\circ F^n\, dm.
\end{split}
\]
Since $\vf_n$ is constant on the elements $Y_{n,j}$ of $\cY_n$ and  $\vf_n\geq n$, we have $|z^{\vf_n}|\leq |z|^n$. The same is true for  $z^{\vf_n\circ F^n}$,
that is $|z^{\vf_n\circ F^n}|\leq |z|^n$. Let $\phi\in C^1$ such that $|\phi |_{C^1(W)}\leq 1$ and compute that
\[
\begin{split}
\left|\int_W (R(z)^n h)\phi\, dm\right|&\leq 
\sum_{W_j\in\mathcal{W}} \left|\int_{W_j} z^{\vf_n} h \phi\circ  F^n\frac{J_{W_j}F^n}{|DF^n|} \,dm\right|\\
&\leq \sum_{Y_{n,j}\in\cY_n}\|h\|_{\cB_w}\left|\phi\circ  F^n\frac{J_{W_j}F^n}{|DF^n|}\right|_{C^1(W_j)} |z|^n .
\end{split}
\]
{\bf Weak norm:} 
W.r.t.\ the $C^0$ norm, we have
\[
\Big|\frac{J_{W_j}F^n}{|DF^n|}\phi\circ  F^n\Big|_{C^0(W_j)}\leq |\phi\circ  F^n|_{C^0(W_j)}\Big||DF^n|^{-1}J_{W_j}F^n\Big|_\infty.
\]
W.r.t.\ the $C^1$ norm, we have
\begin{align*}
\Big|\frac{J_{W_j}F^n}{|DF^n|}\phi\circ  F^n\Big|_{C^1(W_j)}&\leq 
||DF^n|^{-1}J_{W_j}F^n|_{C^1(W_j)}|\phi\circ  F^n|_{C^0(W_j)}\\
&\quad + ||DF^n|^{-1}J_{W_j}F^n |_{C^0(W_j)}|\phi\circ  F^n|_{C^1(W_j)}.
\end{align*}
By definition,
\[
|\phi\circ F^n|_{C^1(W_j)}=\sup_{z\in W_j} |\phi\circ F^n(z)|+
\sup_{z,w\in W_j}\frac{|\phi\circ F^n(z)-\phi\circ F^n(w)|}{d(F^n(z),F^n(w))}.
\]
But for $z,w\in W_j$, 
\begin{equation}\label{eq-Cphi1}
\frac{|\phi\circ F^n(z)-\phi\circ F^n(w)|}{d(F^n(z),F^n(w))}\frac{d(F^n(z),F^n(w))}{d(z,w)}
\leq C|\phi|_{C^1(W)} |J_{W_j}F^n|_{C^1(W_j)}\leq \lambda^{-n} |\phi|_{C^1(W)},
\end{equation}
where in the last inequality we have used~\eqref{eq-dist1}. Hence, $|\phi\circ  F^n|_{C^1(W_j)}\leq C  |\phi|_{C^1(W_j)}$.
Putting the above together,
\begin{equation}\label{eq:weak-ly-test}
\Big|\frac{J_{W_j}F^n}{|DF^n|}\phi\circ  F^n\Big|_{C^1(W_j)}\leq C \Big||DF^n|^{-1}J_{W_j}F^n\Big|_\infty
\, | \phi|_{C^1(W)}.
\end{equation}
As a consequence, we can estimate the weak norm as follows:
\begin{align*}\label{eq:weak-ly}
\left|\int_W (R(z)^n h)\phi\, dm\right|\leq C\|h\|_{\cB_w}|z|^n\sum_j \Big||DF^n|^{-1}J_{W_j}F^n\Big|_\infty\leq C\|h\|_{\cB_w}|z|^n,
\end{align*}
where we have used~\eqref{eq-dist3}. This ends the proof of the first claimed inequality in $\cB_w$.
\\[2mm]
{\bf Strong stable norm:} 
Given  $|\phi|_{C^q(W)}\leq 1$, we have
\begin{equation}
\label{eq-stab}
\begin{split}
\left|\int_W (R(z)^n h)\phi\, dm\right|&\leq \sum_j \left|\int_{W_j} h \frac{J_{W_j}F^n}{|DF^n|}\phi\circ  F^n\,dm\right|\,|z|^n\\
&\leq  \sum_j \left|\int_{W_j} h \hat \phi_j\,dm\right|\,|z|^n+\left|\int_{W_j} h \frac{J_{W_j}F^n}{|DF^n|}\bar \phi_j\,dm\right|\,|z|^n,
\end{split}
\end{equation}
where
\[
\hat \phi_j=\frac{J_{W_j}F^n}{|DF^n|}(\phi\circ  F^n-\bar\phi_j), \qquad \bar \phi_j=|W_j|^{-1}\int_{W_j}\phi\circ  F^n\, dm.
\]
Proceeding as in~\eqref{eq-Cphi1},
\begin{align*}
\frac{\phi\circ  F^n(z)-\phi\circ  F^n(w)}{d(z,w)^q}\leq  C|\phi|_{C^q(W)} |J_{W_j}F^n|_{C^1(W_j)}^q.
\end{align*}
Thus, if $\text{H\"ol}_q(\phi)$ is the H\"older constant of $\phi$, then the distortion bounds \eqref{eq-distort1}
together with the above inequality imply that
\[
\text{H\"ol}_q(\phi\circ F^n)\leq C |J_{W_j}F^n|_{C^0(W_j)}^q \text{H\"ol}_q(\phi).
\]
As a consequence, using \eqref{eq-dist1} we obtain
\begin{align*}
|\phi\circ  F^n-\bar\phi_j|_{C^0(W_j)}&\leq |\sup_{W_j}\phi\circ  F^n-\inf_{W_j}\phi\circ  F^n|
\leq \text{H\"ol}_q(\phi\circ F^n)|W_j|^q\\
&\leq C |J_{W_j}F^n|_{C^0(W_j)}^q \text{H\"ol}_q(\phi)\leq C |J_{W_j}F^n|_\infty^q\leq C\lambda^{-nq}.
\end{align*}
Next, note that
\begin{align*}
|\hat \phi_j|_{C^q(W_j)}& \leq |\phi\circ  F^n-\bar\phi_j|_{C^q(W_j)}||DF^n|^{-1}J_{W_j}F^n|_{C^0(W_j)}\\
&\quad + |\phi\circ  F^n-\bar\phi_j|_{C^0(W_j)}|\, |DF^n|^{-1}J_{W_j}F^n|_{C^q(W_j)}\\
&\leq 2|\phi\circ  F^n-\bar\phi_j|_{C^0(W_j)}|\, |DF^n|^{-1}J_{W_j}F^n|_\infty.
\end{align*}
Putting the above together,
\[
|\hat \phi_j|_{C^q(W_j)}\leq  C\lambda^{-nq} ||DF|^{-1}J_{W_j}F^n|_\infty,
\]
which together with~\eqref{eq-stab} implies that
\[
\left|\int_W (R(z)^n h)\phi\, dm\right|\leq C\|h\|_s\lambda^{-nq}\,|z|^n +C\|h\|_{\cB_w}\,|z|^n.
\]
{\bf Strong unstable norm:} 
Let $W,\tilde W\in \Sigma\cap Y_\ell$. Using the chart $\chi_\ell$ there is a natural bijection
\begin{equation}\label{eq:v}
v: \tilde W \to W, \qquad v \circ \chi_\ell(\tilde g(\eta), \eta)) = \chi_\ell(g(\eta), \eta)).
\end{equation}
Write $F^{-n}(W)=\cup_{j\in\bN} W_j$, 
$F^{-n}(\tilde W)=\cup_{j\in\bN} \tilde W_j$ with $W_j, \tilde W_j \subset Y_{n,j}$
and let $v_j:\tilde W_j \to W_j$, $y \mapsto F^{-n} \circ v \circ F^n(y)$ be the corresponding
bijection between the preimage leaves.

Let $\phi:W \to \R$ and $\tilde \phi:\tilde W \to \R$ be such that $|\phi|_{C^1(W)}, |\tilde\phi|_{C^1(\tilde W)} \leq 1$.
Define $\tilde \psi:\tilde W \to \R$ by
\begin{equation}\label{eq:psi}
\tilde\psi(y) = \phi(v(y)) \frac{ (J_{W_j}F^n \, |DF^n|^{-1}|) \circ F^{-n}(v(y))}
{ (J_{\tilde W_j}F^n \, |DF^n|^{-1}|) \circ F^{-n}(y)}.
\end{equation}
This choice of $\tilde\psi$ is such that
$$
d( J_{W_j}F^n \, |DF^n|^{-1}\, \phi \circ F^n, J_{\tilde W_j}F^n \, |DF^n|^{-1}\, \tilde\psi \circ F^n) = 0,
$$
so that normalising $\phi$ and $\tilde \psi$ by some factor doesn't change the distance
between these quantities.

Compute that
\begin{align*}
\Big|\int_{W} R(z)^nh\phi\, dm&-\int_{\tilde W} R(z)^nh\phi\, dm\Big| \\
\le& |z|^n \sum_j \Big| \int_{W_j}h\frac{J_{W_j}F^n}{|DF^n|}\phi\circ  F^n \, dm
-\int_{\tilde W_j} h\frac{J_{W_j}F^n}{|DF^n|}\tilde\psi\circ  F^n \, dm\Big| \\
&+ |z|^n \sum_j \Big|\int_{\tilde W_j} \frac{J_{W_j}F^n \circ v_j}{|DF^n|}
h (\tilde\phi - \tilde\psi)\circ  F^n\Big| \, dm = S_1+S_2.
\end{align*}
Now by \eqref{eq-dist1} and \eqref{eq:weak-ly-test},
\begin{align*}
S_1\leq \|h\|_u |z|^n \sum_j d(W_j,\tilde W_j) | |DF^n|^{-1} J_{W_j}F^n|_\infty
\leq  C \lambda^{-n} \|h\|_u |z|^n \sum_j d(W_j,\tilde W_j).
\end{align*}
Recall $L > 0$ from \eqref{eq:L} and that $L_u(Y_{n,j})$ is the largest unstable length of $Y_{n,j}$ 
and note that 
\[
d(W_j,\tilde W_j)\leq  C \frac{L_u(Y_{n,j})}{L} d(W,\tilde W).
\]
Hence, renaming $C/L$ to $C$,
\begin{align*}
S_1\leq C \lambda^{-n} \|h\|_u  |z|^n  \sum_j L_u(Y_{n,j})\, d(W,\tilde W) 
\leq C\|h\|_u  |z|^n d(W,\tilde W) .
\end{align*}
For $S_2$, using the definition of $\tilde\psi$, we split 
$$
(\tilde \phi - \tilde\psi) \circ F^n = (\tilde \phi - \phi \circ v) \circ F^n
+ \Big( \frac{|DF^n|^{-1}  J_{W_j} F^n \circ v_j}
{|DF^n|^{-1}  J_{\tilde W_j} F^n} - 1 \Big)  \phi \circ v \circ F^n.
$$
For the first term, \eqref{eq:weak-ly-test} gives
\begin{equation}\label{eq:w1}
| J_{\tilde W_j} F^n |DF^n|^{-1} (\tilde \phi - \tilde\psi) \circ F^n |_{C^1(\tilde W_j)}
\leq C | J_{\tilde W_j} F^n |DF^n|^{-1}|_\infty |\tilde \phi - \phi \circ v|_{C^1(\tilde W)}
\end{equation}
and $|\tilde \phi - \phi \circ v|_{C^1(\tilde W)} =  d(\tilde \phi, \phi) \leq d(W, \tilde W)$. 
Using the weak norm and summing over $j$ we get
$$
\sum_j \Big|\int_{\tilde W_j} \frac{J_{W_j}F^n \circ v_j}{|DF^n|}
h (\tilde\phi - \phi \circ v)\circ  F^n\Big| 
\leq C \| h \|_{\cB_w} d(W, \tilde W).
$$
Finally, since $\Big( \frac{|DF^n|^{-1}  J_{W_j} F^n \circ v_j}
{|DF^n|^{-1}  J_{\tilde W_j} F^n} - 1\Big) \leq C' d(W,\tilde W)$ by \eqref{eq-distort1},
using the weak norm and summing again over $j$ we get
\begin{equation}\label{eq:w2}
\sum_j \int_{\tilde W_j} \frac{J_{W_j}F^n \circ v_j}{|DF^n|} h \phi \circ v \circ  F^n 
\Big| \frac{|DF^n|^{-1}  J_{W_j} F^n \circ v_j} {|DF^n|^{-1}  J_{\tilde W_j} F^n} - 1 \Big|
\leq C \| h \|_{\cB_w} d(W, \tilde W).
\end{equation}
Thus, $S_2 \leq 2C\|h\|_{\cB_w} |z|^n d(W,\tilde W)$.
Putting together the estimates for $S_1$ and $S_2$,
\[
\|R(z)^n h\|_{u} \leq C\|h\|_{\cB} \lambda^{-n}|z|^n + C \| h \|_{\cB_w}|z|^n,
\] 
and the conclusion follows.~\end{proof}

 Proposition~\ref{prop-LSinequality} and Lemma~\ref{lem:R-well-def} imply that $R(z)\in L(\cB,\cB)$, i.e., 
 Hypothesis (H1)(ii) holds true. 
As in~\cite{LT}, we note that Proposition~\ref{prop-LSinequality} alone would not suffice.
The fact that a function has a bounded norm does not imply that it belongs to $\cB$: for this, it is necessary to prove that it can be approximated  by $C^1$ functions in the topology of the Banach space.

The proof of Lemma~\ref{lem:R-well-def} holds essentially unchanged also for the operator $R(z)$, 
thus $R(z)\in L(\cB,\cB)$. We can then extend, by denseness, the statement of 
Proposition~\ref{prop-LSinequality} to all $h\in \cB$, whereby proving hypothesis (H5)(i).

\subsection{Verifying (H1)(iv) and (H5)(ii)}

The following result is an immediate consequence of Lemma~\ref{prop-LSinequality} and the compact embedding stated 
in Lemma~\ref{lemma-embed} (e.g.\ see \cite{Hen}).
\begin{lemma}\label{lem:compactenss}
For each $z\in\overline\bD$ the operator $R$ is quasi-compact with spectral radius bounded by $|z|$ and 
essential spectral radius bounded by $|z|\lambda^{-q }$.
\end{lemma}

Note that $1$ belongs to the spectrum $\sigma(R)$ of $R$ (since the composition with $F$ is the 
dual operator to $R$ and $1\circ F=1$). By the spectral decomposition of $R$ it follows that 
$\frac 1n\sum_{i=0}^{n-1} R^i$ converges (in uniform topology) to the eigenprojector $P$ associated 
to the eigenvalue $1$. Recall that $\mu=P1$. 

The result below gives the characterization of the peripheral spectrum and it goes 
the same with both statement and proof as in~\cite{LT}.
\begin{lemma}{~\cite[Lemma 7.6]{LT} }
\label{lem:peripheral}
Let $\nu\in\sigma(R(z))$ with $|\nu|=1$. Then any associated eigenvector $h$ is a complex measure. Moreover, 
such measures are all absolutely continuous with respect to $\mu$ and have bounded Radon-Nikod\'ym derivatives.
\end{lemma}

Now suppose that $R h=e^{i\theta}h$. Then, by Lemma~\ref{lem:compactenss}, there exists 
$v\in L^\infty(\mu)$ such that $h=v\mu$. Hence
\[
\mu (v\phi)=h(\phi)=e^{-i\theta}h(\phi\circ F)=e^{-i\theta}\mu( v\phi\circ F)=e^{-i\theta}\mu(\phi v\circ F^{-1})
\]
implies $v=e^{i\theta}v\circ F\,$ $\mu$-almost surely. By similar arguments, 
if $z=e^{i\theta}$ and $R(z)h=R(e^{i\theta\vf}h)=h$, then there exists $v\in L^\infty(\mu)$ such 
that $ve^{i\theta \varphi}=v\circ F\,$ $\mu$-a.e.

\begin{prop} Hypotheses (H1)(iv) and (H5)(ii) hold true. 
\end{prop}

\begin{proof} As the proof of the two hypotheses is essentially the same, we limit ourselves to the proof of (H5)(ii).
The proof below is a slight modification of~\cite[Proof of Proposition 7.9]{LT}, replacing  the map $F$ there with the quotient
map $\bar F$ of the map $F$ used here. We recall that $\bar F:\bar Y\to\bar Y$, where $\bar Y=Y/\sim$ with
$x\sim y$ if $x,y$ are on the same stable leaf $W^s\in\Sigma$ (more needed facts about $\bar F$   are recalled and used below).
The idea is to show that a negation of (H5)(ii) for $F$ leads to a similar statement for $\bar F$, which is known to be a contradiction,
since $\bar F$ is Gibbs Markov\footnote{see \cite[page 198]{AaronsonDenker01} for the definition of Gibbs-Markov maps}(this is recalled at the end of the present argument).

 Let $v:Y\to\C$ be a (non identically zero) measurable solution to the equation  $v\circ F=e^{i\theta\varphi}v$ 
 $\mu$-a.e.\ on $Y$, with
$\theta\in (0, 2\pi)$. By Lusin's theorem, $v$ can be approximated in $L^1(\mu)$ by a $C^0$ function, which in turn can be approximated by a
$C^\infty$ function. Hence, there exists a sequence $\xi_n$ of $C^1$ functions  such that $|\xi_n-v|_{L^1(\mu)}\to 0$, as $n\to\infty$. So,
we can write 
\[
v=\xi_n+\rho_n,
\] where $\|\rho_n\|_{L^1(\mu)}\to 0$, as $n\to\infty$.

Starting from $v=e^{-i\theta\varphi} v\circ F$ and iterating forward $m$ times 
(for some $m$ large enough to be specified later),
\begin{align*}
v=e^{-i\theta\sum_{j=0}^{m-1}\varphi\circ F^j}( \xi_n\circ F^m+\rho_n\circ F^m).
\end{align*}

Clearly,
\begin{align}\label{eq-rhon}
|e^{-i\theta\sum_{j=0}^{m-1}\varphi\circ F^j}\rho_n\circ F^m|_{L^1(\mu)}= |\rho_n\circ F^m|_{L^1(\mu)}= |\rho_n|_{L^1(\mu)}\to 0,
\end{align}
as $n\to\infty$.

Next, put $A_{n,m}:=e^{-i\theta\sum_{j=0}^{m-1}\varphi\circ F^j} \xi_n\circ F^m$ and let $d^s$ be 
the derivative in the stable direction. Note that for all $n$ and $m$
\[
|d^s A_{n,m}|\leq |d^s \xi_n|_{\infty}|d^s F^m|.
\]

By~\eqref{eq-dist}, there exists $0<\tau<1$ such that $|d^s F|=\tau$.
Hence, for any $\ve>0$ and any
$n\in\bN$, there exists $m\in\bN$ such that
\[
|d^s A_{n,m}|_\infty<\varepsilon.
\]
It is then convenient to use $\bE_\mu$ for the expectation with respect to $\mu$ and $\bE_\mu(\cdot\;|\; x)$ for the conditional expectation with respect to the $\sigma$-algebra generated by the set of admissible leaves.

Next, we note that the point $\bar x\in\bar Y$ can be written as $\bar x= W^s(x)$. 
With this notation, the previous displayed equation implies that
$|A_{n,m}(x,y)-\bE_\mu(A_{n,m}\;|\;\bar x)|\leq \varepsilon$.
For arbitrary $\psi\in L^\infty(\mu)$, we can then write

\begin{equation*}\label{eq-psiA}
\begin{split}
\bE(\psi v)&=\bE_\mu(\psi A_{n,m})+O(\ve)=\bE_\mu(\psi \bE_\mu(A_{n,m}\;|\;\bar x))+O(\varepsilon)\\
&=\bE_\mu( \bE_\mu(\psi\;|\;\bar x)A_{n,m})+O(\varepsilon)=\bE_\mu(\psi \bE_\mu( v\;|\;\bar x))+O(\ve).
\end{split}
\end{equation*}
Since $\ve$ and $\psi$ are arbitrary, it follows that $v=\bE_\mu( v\;|\;\bar x)$. 

Finally, recall the projection map $\pi: Y\to \bar Y$, $\pi\circ F=\bar F\circ \pi$ and write $\bar\varphi\circ\pi=\varphi$.
Since $v=\bE_\mu( v\;|\;\bar x)$, there exists $\bar v:\bar Y\to \C$ such that $\bar v\circ \pi$. Thus,
$v\circ F=e^{i\theta\varphi}v$ is equivalent to $\bar v\circ \bar F\circ\pi=e^{i\theta\bar\varphi\circ\pi}\bar v\circ\pi$.
So, $\bar v\circ \bar F=e^{i\theta\bar\varphi}\bar v$. But since $\bar F$ is Gibbs Markov, this
equation has only the trivial solution $v=0$ (see~\cite[Theorem 3.1]{AaronsonDenker01}).
\end{proof}

\subsection{Verifying (H4): bounds for $\|R_n\|_{\cB}$.}

\begin{lemma}\label{lem:asymptotic}
There is $C > 0$ such that $\|R_n \|_{\cB}\leq C m(\vf=n)$ for all $n \in \N$.
\end{lemma}

\begin{proof}
Note that $\vf$ is constant on each element of $\cY_1$. 
Thus, given $W\in\Sigma$ we can index the leaves in $F^{-1}(W)$ as $W_n = F^{-1}(W) \cap \{ \vf = n\}$.
Let $\phi\in C^q$ be such that $|\phi|_{C^q(W)}\leq 1$.
 Using~\eqref{eq-dist3} we compute that
\begin{align*}
\Big|\int_W (R_n h)\phi\, dm\Big|&\leq 
 \left|\int_{W_n} h |\, |DF^n|^{-1} J_{W_n}F|\phi\circ  F\,dm\right|\\
&\leq \|h\|_s \left||DF^n|^{-1} J_{W_n} F \right|_{C^q(W_n)} \leq C m(\vf=n).
\end{align*}
Hence, $\|R_n \|_{s}\leq C m(\vf=n)$. By the same argument (with $q=1$), 
$\|R_n \|_{\cB_w} \leq C m(\vf=n)$.

The estimate for the unstable norm follows by the argument used in the proof of 
Proposition~\ref{prop-LSinequality} (for estimating the unstable norm). 
We sketch the argument for completeness.
Let $W,\tilde W\in \Sigma \cap Y_j$, for some $j\geq 1$. 
As above, write $W_n=F^{-1}(W)\cap \{\vf=n\}$ and $\tilde W_n = F^{-1}(\tilde W) \cap \{\vf=n\}$.
Let $|\phi|_{C^1(Y_j)} \leq 1$. Compute that (with $v_n$ and $\tilde\psi$
analogous to \eqref{eq:v} and \eqref{eq:psi})
\begin{align*}
\Big|\int_{W} (R_nh)\phi\, dm -\int_{\tilde W} (R_nh)\phi\, dm\Big| 
\le& \Big| \int_{W_n}h\frac{J_{W_n}F}{|DF|}\phi\circ  F \, dm
-\int_{\tilde W_n} h\frac{J_{W_n}F}{|DF|}\tilde\psi\circ  F \, dm\Big| \\
&+ \Big|\int_{\tilde W_n} \frac{J_{W_n}F \circ v_n}{|DF|}
h (\tilde\phi - \tilde\psi)\circ  F\Big| \, dm = S_1+S_2.
\end{align*}
By the argument used  in the proof of Proposition~\ref{prop-LSinequality} together with \eqref{eq-dist3},
\begin{align*}
S_1 &\leq \|h\|_u d(W_n,\tilde W_n) | |DF|^{-1} J_{W_n}F|_\infty
\leq C \|h\|_u m(\vf=n)d(W, \tilde W),
\end{align*}
where we have used that $\frac{d(W_n,\tilde W_n)}{d(W, \tilde W)} \approx \frac{L_u(Y_n)}{L} 
\leq Cm(\vf = n)$, the last inequality being guaranteed by~\eqref{eq-dist3}.
Finally, using the analogues of \eqref{eq:w1} and \eqref{eq:w2} (for a single term in each), we find
\begin{eqnarray*}
\int_{\tilde W_n} \frac{J_{W_n}F \circ v_n}{|DF|} h (\tilde\phi - \tilde\psi)\circ  F \, dm
&\leq& C \| h \|_{\cB_w} |  J_{\tilde W_n}F |DF|^{-1}|_\infty |\tilde \phi-\phi \circ v|_{C^1(\tilde W)}\\
&&+ \, C  \| h \|_{\cB_w} |  J_{\tilde W_n}F |DF|^{-1}|_\infty d(W,\tilde W)\\
\leq& 2C  \| h \|_{\cB_w} \mu(\varphi = n) d(W, \tilde W).
\end{eqnarray*}
Combining these estimates we find $S_2\leq C\|h\|_{cB_w} m(\vf=n) d(W,\tilde W)$, 
ending the proof.~\end{proof}


\begin{thebibliography}{DRR99}

\bibitem[AD01]{AaronsonDenker01} J.~Aaronson, M.~Denker, 
{Local limit theorems for partial sums of stationary sequences generated by Gibbs-Markov maps}. 
\emph{Stoch. Dyn.} \textbf{1} (2001) 193--237.

\bibitem[AA13]{AA13} J.\ Alves, D.\ Azevedo, 
{Statistical properties of diffeomorfisms with weak invariant manifolds,}
Preprint 2013, arXiv:1310.2754 

\bibitem[BCD11]{BCD11} P.\ B\'alint, N.\ Chernov, D.\ Dolgopyat,
{Limit theorems for dispersing billiards with cusps},
\emph{Commun.\ Math.\ Phys.\ } {\bf 308} (2011) 479-510.

\bibitem[DL08]{DemersLiverani08} M.~Demers, C.~Liverani, 
{Stability of statistical properties in two dimensional piecewise hyperbolic maps}.
\emph{Trans.\ Amer.\ Math.\ Soc.\ } \textbf{360} (2008) 4777--4814.



\bibitem[DR83]{DR} F.\ Dumortier, R.\ Roussarie, 
{\em Germes de diff\'eomorphismes et de champs de vecteurs en classe de différentiabilit\'e finie,}
Ann.\ Inst.\ Fourier {\bf 33} (1983), no. 1, 195--267.

\bibitem[DRR81]{DRR} F.\ Dumortier, P.\ Rodrigues, R.\ Roussarie, 
{Germs of diffeomorphisms in the plane,}
\emph{Lect.\ Notes in Math.\ } {\bf 902} Springer-Verlag, Berlin-New York, 1981. iv+197 pp.

\bibitem[GL06]{GL06} S.\ Gou{\"{e}}zel, C.\ Liverani,
{Banach spaces adapted to {A}nosov systems},
\emph{Ergodic Theory and Dynamical Systems}, {\bf 26}(1):189--217, 2006.

\bibitem[G04]{Gouezel04} S.~Gou{\"e}zel, 
{Sharp polynomial estimates for the decay of correlations}.
\emph{Israel J.\ Math.} \textbf{139} (2004) 29--65.

\bibitem[G11]{Gouezel11} S.~Gou{\"e}zel, 
{Correlation asymptotics from large deviations in dynamical systems with infinite measure,}
\emph{Colloquium Math.} \textbf{125}, (2011) 193--212.

\bibitem[H93]{Hen} H.\ Hennion,  
{Sur un th\'eor\`eme spectral et son application aux noyaux lipchitziens,}  
\emph{Proc.\ Amer.\ Math.\ Soc.\ } {\bf 118} (1993) 627--634.

\bibitem[H00]{Hu00} H.~Hu, 
{Conditions for the existence of SBR measures of ``almost Anosov'' diffeomorphisms,} 
\emph{Trans.\ Amer.\ Math.\ Soc.\ } \textbf{352} (2000) 2331--2367.

\bibitem[HY95]{HY95} H.\ Hu,  L-S.\ Young, 
{Nonexistence of SBR measures for some diffeomorphisms that are "almost Anosov",}
\emph{Ergodic Theory and Dynamical Systems,} {\bf 15}, (1995) 67--76.

\bibitem[HZ16]{HuZhangprepr} H.\ Hu, X.\ Zhang,
{Polynomial decay of correlations for almost Anosov diffeomorphisms,}
Ergod.\ Th.\ Dynam.\ Sys.\ published online  2017 doi:10.1017/etds.2017.45

\bibitem[K79]{K79} A.\ Katok, 
{\em Bernoulli diffeomorphisms on surfaces,}
Ann.\  of  Math.\  {\bf 110} (1979), 529--547.

\bibitem[LM05]{LM05} C.\ Liverani, M.\ Martens,
{Convergence to equilibrium for intermittent symplectic maps,}
Commun.\ Math.\ Phys., {\bf 260}, (2005) 527--556. 

\bibitem[LT16]{LT} C.\ Liverani, D.\ Terhesiu, 
{Mixing for some non-uniformly hyperbolic systems},
\emph{Annales Henri Poincar\'e}, {\bf 17}, no. 1, (2016) 179-226.

\bibitem[M83]{Machta83} J.\ Machta,
{\em Power law decay of correlations in a billiard problem,} 
\emph{J.\ Statist.\ Phys.\ } {\bf 32} (1983), 555-564.

\bibitem[M15]{Mel15} I.\ Melbourne,
Mixing for invertible dynamical systems with infinite measure, 
\emph{Stoch.\ Dyn.\ } {\bf 15} (2015) 1550012 (25 pages). 

\bibitem[MT12]{MT} I.\ Melbourne, D.\ Terhesiu, 
{Operator renewal theory and mixing rates for dynamical systems with infinite measure},  
\emph{Invent.\ Math.} \textbf{1} (2012) 61--110.

\bibitem[P74]{Palis} J.\  Palis, 
{Vector fields generate few diffeomorphisms,} 
\emph{Bull.\ Amer.\ Math.\ Soc.\ } {\bf 80} (1974), no. 3, 503--505.

\bibitem[PSZ16]{PSZ16} Y.\ Pesin, S.\ Senti, K.\ Zhang,
{Thermodynamics of towers of hyperbolic type,}
Trans.\ Amer.\ Math.\ Soc.\ {\bf 368} (2016) 8519--8552. 

\bibitem[S02]{Sarig02} O.~M. Sarig, 
{Subexponential decay of correlations}. 
\emph{Invent.\ Math.\ } \textbf{150} (2002) 629--653.

\bibitem[T12]{Teschl} G.\ Teschl,
{\em Ordinary Differential Equations and Dynamical Systems,}
Graduate Studies in Mathematics, {\bf 140}, Amer.\ Math.\ Soc., Providence, 2012.

\bibitem[T16]{T-PTRF} D.\ Terhesiu,
{Mixing rates for intermittent maps of high exponent,} 
\emph{Probability Theory and Related Fields,} {\bf 166} (2016), no. 3-4, 1025--1060.

\end{thebibliography}
\end{document}